\documentclass{amsart}
    \usepackage{amsmath}
    \usepackage{amssymb,amsfonts}
    \usepackage[all,arc]{xy}
    \usepackage{caption}
    \usepackage{enumerate}
    \usepackage{mathrsfs}
    \usepackage{amsthm}
    \usepackage[shortlabels]{enumitem}
    \usepackage{extarrows}
    \usepackage{verbatim}
    \usepackage{dsfont}
    \usepackage{tikz-cd}
    \usepackage{tensor}
    \usepackage{mathtools}
    \usepackage{upgreek}
    \usepackage{xcolor}
    \usepackage{extarrows}
    \usepackage{mathtools}
    \usepackage{tikz}
    \usepackage{pgfplots}
    \usetikzlibrary{positioning}
    \usetikzlibrary{calc}

    \usepackage{hyperref}

    \newtheorem{thm}{Theorem}[section]
    \newtheorem{cor}[thm]{Corollary}
    \newtheorem{prop}[thm]{Proposition}
    \newtheorem{lem}[thm]{Lemma}

        \newtheorem{lemdef}[thm]{Lemma/Definition}
    
    \theoremstyle{definition}
    \newtheorem{defn}[thm]{Definition}

    \theoremstyle{remark}
    
    \newtheorem{rem}[thm]{Remark}

    \newcommand{\Z}{\mathbb{Z}}
    \newcommand{\R}{\mathbb{R}}
    \newcommand{\C}{\mathbb{C}}
    
    \newcommand{\Ecal}{\mathcal{E}}
    \newcommand{\Fcal}{\mathcal{F}}

    \newcommand{\Ical}{\mathcal{I}}
    \newcommand{\Lcal}{\mathcal{L}}
    \newcommand{\Ocal}{\mathcal{O}}
    \newcommand{\Pcal}{\mathcal{P}}
    \newcommand{\Qcal}{\mathcal{Q}}
    
    \newcommand{\Ucal}{\mathcal{U}}
    \newcommand{\Wcal}{\mathcal{W}}
    \newcommand{\Xcal}{\mathcal{X}}
    \newcommand{\Zcal}{\mathcal{Z}}

    \newcommand{\Pbb}{\mathbb{P}}
    \newcommand{\Qbb}{\mathbb{Q}}
    
    \newcommand{\Esc}{\mathscr{E}}
    \newcommand{\Fsc}{\mathscr{F}}
    
    \newcommand{\Isc}{\mathscr{I}}
    \newcommand{\Usc}{\mathscr{U}}

    \newcommand{\abf}{\mathbf{a}}
    \newcommand{\bbf}{\mathbf{b}}
    \newcommand{\vbf}{\mathbf{v}}

    \newcommand{\Bl}{\mathrm{Bl}}

    \newcommand{\Coh}{\mathrm{Coh}}

    \newcommand{\NS}{\mathrm{NS}}

    \newcommand{\Supp}{\mathrm{Supp}}
     \newcommand{\supp}{\mathrm{supp}}
    \newcommand{\Hom}{\mathrm{Hom}}
    \newcommand{\RHom}{R\mathcal{H}om}
    \newcommand{\Ext}{\mathrm{Ext}}
    
    \newcommand{\RExt}{\mathcal{E}xt}
    \newcommand{\ch}{\mathrm{ch}}
    \newcommand{\Pic}{\mathrm{Pic}}

    \newcommand{\hilb}{S^{[10]}}
    \newcommand{\BN}{\mathrm{BN}}
    \newcommand{\Mov}{\mathrm{Mov}}
    \newcommand{\Halg}{H^*_{\mathrm{alg}}}
    \newcommand{\Stab}{\mathrm{Stab}}

    \makeatletter
    
    \newcommand*{\rom}[1]{\expandafter\@slowromancap\romannumeral #1@}
    
    \makeatletter
    
    \let\c@equation\c@thm
    \makeatother
    \numberwithin{equation}{section}

\title[Birational geometry of Beauville-Mukai systems I]{Birational geometry of Beauville-Mukai systems I: the rank three and genus two case}
\author{Xuqiang Qin and Justin Sawon}
\address{Department of Mathematics, University of North Carolina, Chapel Hill, NC 27514, USA}
\email{qinx@unc.edu, sawon@email.unc.edu}

\date{July 2022}

\pgfplotsset{compat=1.16}

\begin{document}

\begin{abstract}
We study wall-crossing for the Beauville-Mukai system of rank three on a general genus two K3 surface. We show that such a system is related to the Hilbert scheme of ten points on the surface by a sequence of flops, whose exceptional loci can be described as Brill-Noether loci. We also obtain Brill-Noether type results for sheaves in the Beauville-Mukai system.
\end{abstract}

\maketitle

\tableofcontents

\section{Introduction}

Irreducible holomorphic symplectic manifolds are generalizations of K3 surfaces. The first higher-dimensional example was discovered by Fujiki~\cite{Fuj83} (the Hilbert scheme $S^{[2]}$ of two points on a K3 surface $S$), and soon generalized by Beauville~\cite{Bea83} (the Hilbert scheme $S^{[n]}$ of $n\geq 2$ points on $S$). An interesting question is: when is an irreducible holomorphic symplectic manifold birational to a Lagrangian fibration, i.e., a fibration by complex tori that are Lagrangian with respect to the holomorphic symplectic form? The conjectural solution has come to be known as the Hyperk{\"a}hler SYZ Conjecture (see~\cite{Huy03,Saw03,Ver10}). For Hilbert schemes $S^{[n]}$ this question was considered by Markushevich~\cite{Marku06} and the second author~\cite{Saw07}, and a complete answer was given by Bayer and Macr{\`i}~\cite{BM14a,BM14b}, verifying the Hyperk{\"a}hler SYZ Conjecture in this case (see also \cite{Bea99}, \cite{Fu03}, \cite{HT00}, and \cite{IR07} for some special cases). The Lagrangian fibration that arises in this case is the compactified relative Jacobian of a complete linear system of curves on a K3 surface, and is known as the Beauville-Mukai integrable system~\cite{Muk84,Bea91}. In fact, Markushevich~\cite{Marku02} conjectured that any Lagrangian fibration whose fibres are Jacobians of curves must be of this type, i.e., the family of curves must be a complete linear system on a K3 surface. He proved this for genus two curves~\cite{Marku96}, and the second author proved this for genus three, four, and five~\cite{Saw15b}, and for all genera under an additional assumption~\cite{Saw15a}.

In the simplest cases, the Hilbert scheme $S^{[n]}$ is actually isomorphic to a Beauville-Mukai system. However, in most cases they are only birational. The birational map can take the form of a Mukai flop~\cite{Muk84}, or a more complicated stratified elementary transform, as studied by Markman~\cite{Mar01}. The goal of this paper is to understand the birational map in a particular case. 

We start with a general polarized K3 surface $(S,H)$ of genus two, i.e., $\pi: S\to \Pbb^2$ is a double cover ramified over a general smooth sextic curve, $H=\pi^*(\Ocal_{\Pbb^2}(1))$, and $\Pic(S)=\Z [H]$. By the results mentioned above, \cite{Marku06,Saw07,BM14a,BM14b}, for such a K3 surface the Hilbert scheme $S^{[n]}$ is birational to a Beauville-Mukai system if and only if $n=m^2+1$ is a perfect square plus one. Here and throughout Beauville-Mukai systems will be identified with Mukai moduli spaces $M(0,m,k)$ of stable sheaves on $S$ (semistable if $m$ and $k$ are not coprime). Recall that the general element of $M(0,m,k)$ looks like $\iota_*L$ where $\iota:C\hookrightarrow S$ is the inclusion of a curve $C\in |mH|$ into the K3 surface, and $L$ is a line bundle on $C$ of degree $k+m^2$. The Fitting support gives a morphism $M(0,m,k)\rightarrow |mH|\cong\Pbb^{m^2+1}$, and $M(0,m,k)$ is a compactification of the relative Jacobian $\mathrm{Pic}^{k+m^2}(\mathcal{C}/|mH|)$ of the family of curves in the linear system $|mH|$. The birational map
$$S^{[m^2+1]}\dashrightarrow M(0,m,-1)$$
is given by taking a general length $m^2+1$ subscheme $\xi$ of $S$ (consisting of $m^2+1$ distinct points) to the line bundle
$$\mathcal{O}_C(-\xi)\otimes\mathcal{O}_S(mH)|_C$$
on the (unique) curve $C\in |mH|$ passing through the $m^2+1$ points. At the level of the derived category of $S$, this birational map is induced by the autoequivalence $\Phi:\mathrm{D}^b(S)\rightarrow\mathrm{D}^b(S)$ given by the composition of the spherical twist $T_{\mathcal{O}_S(-mH)}$ and tensoring with the line bundle $\mathcal{O}_S(mH)$.

When $m=1$, and thus $n=2$, Mukai~\cite[Example 0.6]{Muk84} observed that the birational map
$$S^{[2]}\dashrightarrow M(0,1,-1)$$
is simply the elementary transform which flops the plane in $S^{[2]}$ parametrizing subschemes $\pi^{-1}(p)$, $p\in\Pbb^2$, to the zero section in $M(0,1,-1)\rightarrow |H|$ parametrizing trivial line bundles $\mathcal{O}_C$. When $m=2$, and thus $n=5$, the birational map
$$S^{[5]}\dashrightarrow M(0,2,-1)$$
was studied by Hellmann~\cite{Hel20}. She decomposed it into flops.
\begin{thm}\cite[Theorem 5.2]{Hel20}
There are five smooth $K$-trivial birational models of $S^{[5]}$, or $M:=M(0,2,-1)$. They are connected by a chain of flopping contractions
\[ \xymatrix@R-1pc@C-2pc{
& \Bl_{W_2}S^{[5]} \ar[dr]\ar[dl] && \Bl_{\tilde{W}_3}X_1\ar[dr]\ar[dl] && \Bl_{\tilde{Z}_3}X_3\ar[dr]\ar[dl] && \Bl_{Z_1}M\ar[dr]\ar[dl]\\
S^{[5]} \ar@{-->}[rr]^{g_1} && X_1 \ar@{-->}[rr]^{g_2} && X_2 \ar@{<--}[rr]^{g_3} &&  X_3 \ar@{<--}[rr]^{g_4} &&M.} \]
\end{thm}
The exceptional loci $W_2\subset W_3\subset S^{[5]}$ and $Z_1\subset Z_3\subset M$ are described explicitly in~\cite{Hel20}; they are either Brill-Noether loci or irreducible components of Brill-Noether loci.

Our paper is greatly inspired by Hellmann's results, and may be seen as an extension to the `next' case. Namely, we consider the case $m=3$, and thus $n=10$, and we decompose the birational map
$$S^{[10]}\dashrightarrow M(0,3,-1)$$
into flops.
\begin{thm}
Let $(S,H)$ be a general polarized K3 surface with $\Pic(S)=\Z [H]$ and $H^2=2$. There are eleven smooth $K$-trivial birational models of $S^{[10]}$, or $M\coloneqq M(0,3,-1)$. They are connected by a chain of flopping contractions
\[ \xymatrix@R-1pc@C-2pc{
& \Bl_{W_0}S^{[10]} \ar[dr]\ar[dl] && \Bl_{\tilde{W}_1}X_1\ar[dr]\ar[dl] && \Bl_{\tilde{W}_2}X_2\ar[dr]\ar[dl] && \Bl_{\tilde{W}_3}X_3\ar[dr]\ar[dl] && \Bl_{\widetilde{W_4\cup \Wcal_0}}X_4\ar[dr]\ar[dl] && \Bl_{\tilde{\Wcal_1}}X_5\ar[dr]\ar[dl]\\
S^{[10]} \ar@{-->}[rr]^{g_0} && X_1 \ar@{-->}[rr]^{g_1} && X_2 \ar@{-->}[rr]^{g_2} && X_3 \ar@{-->}[rr]^{g_3} &&  X_4 \ar@{-->}[rr]^{g_4} &&X_5\ar@{-->}[rr]^{g_5} &&X_6\xrightarrow{\Phi}} \]
\[ \xymatrix@R-1pc@C-2pc{
&  \Bl_{\tilde{Z}_8}X_7\ar[dr]\ar[dl] && \Bl_{\tilde{Z}_{1,3}}X_8\ar[dr]\ar[dl] && \Bl_{\tilde{Z}_4}X_9\ar[dr]\ar[dl] && \Bl_{Z_2}M \ar[dr]\ar[dl]\\
 \xrightarrow{\Phi}X_6' && X_7 \ar@{-->}[ll]^{g_6} && X_8 \ar@{-->}[ll]^{g_7} &&  X_9 \ar@{-->}[ll]^{g_8} &&M. \ar@{-->}[ll]^{g_{9}}} \]
\end{thm}
This is restated in greater detail as Theorem~\ref{main}. In particular, the exceptional loci $W_0 \subset W_1\subset\cdots\subset W_4 \subset S^{[10]}$, $\Wcal_0\subset\Wcal_1\subset \hilb$, $Z_2 \subset Z_4 \subset Z_8\subset M$, and $Z_2\subset Z_{1,3}\subset Z_8\subset M$ are described explicitly there; they are once again Brill-Noether loci or irreducible components of Brill-Noether loci.

To prove this theorem, we employ the powerful techniques developed by Bayer and Macr{\`i}~\cite{BM14a,BM14b} for determining the birational models of a moduli space. The space of Bridgeland stability conditions on $S$ has a wall-and-chamber structure. Computing the walls typically requires some work, as it involves solving Diophantine equations, though in our case these reduce to certain Pell's equations that have already been carefully analyzed by Cattaneo~\cite{Cat19}. For the Hilbert scheme $S^{[10]}=M(1,0,-9)$, we consider moduli spaces $M_{\sigma}({\bf v})$ of $\sigma$-stable objects with Mukai vector ${\bf v}=(1,0,-9)$. If we vary $\sigma$ in a single chamber the $\sigma$-stable objects do not change and the moduli spaces are isomorphic; one of the chambers corresponds to the moduli space of Gieseker stable sheaves, i.e., $S^{[10]}$. When $\sigma$ crosses a wall some $\sigma$-stable objects become unstable, while new $\sigma$-stable objects appear. One can calculate these loci of destabilized objects in the moduli spaces, and thus arrive at a precise description of the resulting birational modification of the moduli space. In our case, there are six walls to cross to reach $X_6$ from $S^{[10]}$. Similarly, for the Mukai vector $(0,3,-1)$ there is a chamber corresponding the Beauville-Mukai system $M:=M(0,3,-1)$, and crossing four walls takes us to $X^{\prime}_6$. Finally, the isomorphism $\Phi:X_6\rightarrow X^{\prime}_6$ is induced the autoequivalence $\Phi$ of the same name mentioned earlier, which takes $\sigma$-stable objects of Mukai vector $(1,0,-9)$ to $\Phi_*(\sigma)$-stable objects of Mukai vector $\Phi_*(1,0,-9)=(0,3,-1)$. For a suitable choice of $\sigma$, we have $X_6=M_{\sigma}(1,0,-9)$ and $X_6'=M_{\Phi_*(\sigma)}(0,3,-1)$.

In some respects, our result is quite similar to Hellmann's. We use the same strategy of analyzing the wall-crossings and birational modifications from both the Hilbert scheme and the Beauville-Mukai system sides, eventually reaching birational models ($X_6$ and $X_6'$, respectively) that can be identified by an isomorphism coming from the autoequivalence $\Phi$. Moreover, as in Hellmann's analysis, all of the walls that we cross are `rank one' (see Remark~\ref{rank_one_walls}), which makes it easier to describe the exceptional loci that are flopped. On the other hand, there are some differences and added complications in our work. By~\cite[Section 14]{BM14b}, the exceptional locus of a flop corresponding to a wall-crossing has a natural stratification given by the decomposition of the Mukai vector. In Hellmann's work these stratifications are always trivial and the flops are simple, or elementary modifications in the terminology of Mukai~\cite{Muk84}. By contrast, we encounter nontrivial stratifications for some wall-crossings. Specifically, the exceptional loci for our fourth, fifth, and sixth wall-crossings have nontrivial stratifications, and the resulting flops are stratified elementary modifications in the terminology of Markman~\cite{Mar01}. A priori, such loci could have several irreducible components, and indeed this occurs for the fourth wall, whose exceptional locus has two irreducible components. These differences make the analysis much more complicated; they also inspired a second paper~\cite{QS22b} where we analyze the general theory in low rank cases.

Regarding the Beauville-Mukai system, note that tensoring with $\Ocal_S(H)$ induces an isomorphism $M(0,3,k)\to M(0,3,k+6)$. Moreover, Hellmann~\cite[Proposition 2.1]{Hel20} proved that
$$\mathcal{E}\mapsto \mathcal{E}^\vee:=\mathcal{E}xt^1(\mathcal{E},\Ocal_S)$$
induces an isomorphism $M(0,3,-1)\to M(0,3,1)$. As a result, the moduli spaces $M(0,3,k)$ are all isomorphic for $k\equiv 1,5 \pmod 6$, and they are all isomorphic for $k\equiv 2,4 \pmod 6$. However, $M(0,3,-1)$ is not isomorphic, nor even birational, to $M(0,3,-2)$, as can easily be seen by applying the argument of~\cite[Proposition 15]{Saw08}. For $k\equiv 0,3 \pmod 6$ the Mukai vector $(0,3,k)$ is divisible by $3$ and the moduli space $M(0,3,k)$ is singular; moreover, Kaledin, Lehn, and Sorger~\cite{KLS06} proved that it does not admit a symplectic desingularization. Thus there are essentially two distinct smooth rank three Beauville-Mukai systems on general genus two K3 surfaces, where {\em rank three\/} is used in analogy with Hitchin systems, see Donagi, Ein, and Lazarsfeld~\cite{DEL97} (also, $M(0,3,k)$ contains sheaves $\iota_*V$ where $V$ is a rank three bundle on a curve in $|H|$). We have chosen to study the birational geometry of $M(0,3,-1)$ because it is birational to the Hilbert scheme $S^{[10]}$.

The paper is organized as follows. In Section~2 we review the preliminaries needed for the rest of the paper. In Section~3 we study the Brill-Noether loci in $\hilb$ and $M(0,3,-1)$. These loci will later be identified with the exceptional loci of wall-crossings. Section~4 contains the main results of this paper: we use the machinery in \cite{BM14b} to identify the walls and give a careful analysis of the wall-crossings. In Section~5 we collect auxiliary wall-crossing results needed for Section~4.

\subsection*{Acknowledgements}
The authors would like to thank Emanuele Macr\`i for helpful discussions. The second author gratefully acknowledges support from the NSF, grant number DMS-1555206.

\subsection*{Notation and convention}
Throughout this paper, $(S,H)$ will be a general polarized K3 surface of genus two, i.e. $\pi: S\to \Pbb^2$ is a double cover ramified over a general smooth sextic curve, with $H=\pi^*(\Ocal_{\Pbb^2}(1))$ and $\Pic(S)=\Z [H]$. We will refer to curves in the linear systems $|H|$, $|2H|$, and $|3H|$ as lines, conics, and cubics, respectively.

We use $\Halg(S,\Z)$ to denote the Mukai lattice $H^0(S,\Z)\oplus \NS(S)\oplus H^4(S,\Z)$ of $S$. We use $\mathrm{D}^b(S)$ to denote the bounded derived category of coherent sheaves on $S$.

Given $\vbf\in\Halg(S,\Z)$, we use $M_H(\vbf)$ and $M^{st}_H(\vbf)$ to denote the moduli spaces of $H$-semistable and $H$-stable sheaves with Mukai vector $\vbf$, respectively. We will often omit the subscript when no confusion will be caused. Given a Bridgeland stability condition $\sigma$, we use $M_\sigma(\vbf)$ and $M^{st}_\sigma(\vbf)$ to denote the moduli spaces of $\sigma$-semistable and $\sigma$-stable objects in $\mathrm{D}^b(S)$ with Mukai vector $\vbf$, respectively.

The Grothendieck group of a triangulated category $\mathcal{D}$ is denoted by $K(\mathcal{D})$.

Let $\Fcal$ be a coherent sheaf on a scheme $X$. We use $\Supp(\Fcal)$ to denote the Fitting support of $\Fcal$. For a closed subscheme $Z\subset X$, we use $\supp(Z)$ to denote the set-theoretic support of $Z$.

\section{Preliminaries}

\subsection{The linear systems $|H|,|2H|$ and $|3H|$.}\label{3H}
Our main object of interest is the Beauville-Mukai integral system $M(0,3,-1)\rightarrow |3H|$, which we will denote by $f:M\rightarrow B$ throughout. We refer the reader to \cite[Section 2.1]{Hel20} for a description of the linear systems $|H|$ and $|2H|$. Regarding $|3H|$, we have 
\begin{align*}
    h^0(S,\Ocal_S(3))=11=h^0(\Pbb^2,\Ocal_{\Pbb^2}(3))+1.
\end{align*}
We can think of the extra dimension as coming from the ramification locus of $S$. Let $B$ denote $|3H|$. Let $\Sigma$ denote the locus of non-integral curves in $B$. Let $\Sigma_{1,1,1}$ denote the locus of curves with three distinct irreducible components in $B$. Let $\Sigma_{2,1}$ denote the locus of curves whose support is of the form $2L_1+L_2$ where $L_i\in|H|$ and $L_1\neq L_2$. Let $\Sigma_{3}$ denote the locus of curves whose support is of the form $3L$ with $L\in|H|$. We  have a chain of closed subschemes 
$$\Sigma_3\subset(\Sigma_{2,1}\cup\Sigma_3)\subset(\Sigma_{1,1,1}\cup \Sigma_{2,1}\cup\Sigma_3)\subset \Sigma\subset B.$$
Let $x\in B$ and use $C$ to denote the corresponding curve.
\begin{itemize}
    \item For $x\in B\backslash\Sigma$, $C$ is an integral cubic curve, and then $f^{-1}(x)$ parametrizes torsion free rank one sheaves on $C$ of degree $8$.
    \item For $x\in \Sigma\backslash(\Sigma_{1,1,1}\cup\Sigma_{2,1}\cup\Sigma_3)$, $C$ is the union of an integral conic $Q$ and a line $L$. Note that $l(Q\cap L)=4$. We have
    \begin{align*}
        0\to \Ical_{Q\cap L/L}\to \Ocal_C\to \Ocal_Q\to 0
    \end{align*}
    and 
    \begin{align*}
        0\to \Ical_{Q\cap L/Q}\to \Ocal_C\to \Ocal_L\to 0.
    \end{align*}
    \item For $x\in \Sigma_{1,1,1}$, $C$ is the union of three distinct lines $L_i$, $i=1,2,3$. We have the following filtration for $\Ocal_C$.
    \[
  \xymatrix{
  0 \ar[r] & \Ical_{((L_2\cup L_3)\cap L_1)/L_1} \ar[r] \ar[d] & \Ical_{(L_3\cap(L_1\cup L_2))/(L_1\cup L_2)} \ar[r] \ar[d] & \Ocal_C \ar[d]\\
  & \Ical_{((L_2\cup L_3)\cap L_1)/L_1} \ar@{.>}[lu] & \Ical_{(L_2\cap L_3)/L_2} \ar@{.>}[lu] &  \Ocal_{L_3}.
  \ar@{.>}[lu] }
  \]
    \item For $x\in \Sigma_{2,1}$, $C$ is the union of a first infinitesimal neighbourhood of a line $L_1$ with a line $L_2\neq L_1$. We have 
    \begin{align*}
        0\to \Ical_{(2L_1\cap L_2)/L_2}\to \Ocal_C\to \Ocal_{2L_1}\to 0
    \end{align*}
    with $0\to \Ical_{L_1}/\Ical^2_{L_1}\to \Ocal_{2L_1}\to \Ocal_{L_1}\to 0$, or equivalently
    \begin{align*}
        0\to \Ical_{(2L_1\cap L_2)/2L_1}\to \Ocal_C\to \Ocal_{L_2}\to 0.
    \end{align*}
    \item For $x\in \Sigma_{3}$, $C$ is the second infinitesimal neighbourhood of a line $L$. We have
    \begin{align*}
        0\to (\Ical_{L}/\Ical^2_{L})^2\to \Ocal_C\to \Ocal_{2L}\to 0
    \end{align*}
    with $0\to \Ical_{L}/\Ical^2_{L}\to \Ocal_{2L}\to \Ocal_{L}\to 0$.
\end{itemize}

\subsection{Zero-dimensional closed subschemes}
Recall that for any closed subscheme $Y\subset S$, we use $\supp(Y)$ to denote the support of $Y$.
\begin{defn}
    Let $\xi\subset S$ be a zero-dimensional closed subscheme. We say a closed subscheme $\xi'\subset \xi$ is \it{saturated} if $\supp(\Ocal_\xi/\Ocal_{\xi'})\cap \supp(\xi')=\emptyset$.
\end{defn}
We collect some notations and easy results about saturated subschemes.
\begin{lemdef}
Let $\xi\subset S$ be a zero-dimensional closed subscheme and $\xi'\subset \xi$ be a saturated subscheme. Then
\begin{enumerate}
    \item $\Ocal_\xi=\Ocal_{\xi'}\oplus(\Ocal_\xi/\Ocal_{\xi'})$, from now on we denote the subscheme of $\xi$ corresponding to the second summand by $\xi\backslash\xi'$.
    \item $\xi\backslash\xi'\subset \xi$ is also a saturated subscheme.
    \item Suppose $\xi''\subset \xi$ is also saturated. Define $\xi'\cup\xi''$ as the saturated subscheme of $\xi$ whose support is $\supp(\xi')\cup\supp(\xi'')$.
    \item For any subscheme $\zeta\subset \xi$, $\zeta\backslash\xi'$ is  defined as the maximal subscheme of $\zeta$ whose support does not intersect $\supp(\xi')$.
\end{enumerate}
\end{lemdef}

\subsection{Stability conditions on  K3 surfaces}\label{pre-stab}
In this section we review some basics facts about (Bridgeland) stability conditions and their moduli spaces on  K3 surfaces. Let $X$ be a K3 surface. Its \emph{algebraic cohomology group} is 
\begin{align*}
    \Halg(X,\Z):=H^0(X,\Z)\oplus \NS(X)\oplus H^4(X,\Z).
\end{align*}
The \emph{Mukai vector} $\vbf:K(\mathrm{D}^b(X))\to \Halg(X,\Z)$ is given by
\begin{align*}
    \vbf(E):=\ch(E)\sqrt{\mathrm{td(X)}}.
\end{align*}
for $E\in K(\mathrm{D}^b(X))$, and the \emph{Mukai pairing} $(-,-)$ on $\Halg(X,\Z)$ is defined by
\begin{align*}
    ((r,c_1,s),(r',c_1',s'))=c_1 c_1'-rs'-r's\in \Z.
\end{align*}
By Riemann-Roch, $\chi(F,F')=-(\vbf(F),\vbf(F'))$ for any $F,F'\in \mathrm{D}^b(X)$. The group $\Halg(X,\Z)$ endowed with the Mukai pairing is called the \emph{algebraic Mukai lattice} of $X$.

Mukai~\cite{Muk84} showed that if $M_H(\vbf)$ is nonempty and smooth, then it has the structure of a hyperk\"ahler manifold. The movable cone of a hyperk\"ahler manifold $\mathcal{M}$ admits a locally polyhedral chamber decomposition~\cite{HT09}, where the chambers corresponds to K-trivial smooth birational models of $\mathcal{M}$ (or hyperk\"ahler manifolds birational to $\mathcal{M}$). For the rest of this paper, we will simply refer to these hyperk\"ahler manifolds as \emph{birational models} of $\mathcal{M}$.

We refer the readers to \cite{Bri07} for the definitions of slicings, hearts, general Bridgeland stability conditions, and the complex manifold structure on the space of stability conditions. We note that, for the rest of this paper, all stability conditions on any K3 surface $X$ will be with respect to the lattice $\Halg(X,\Z)$. We now review the notion of geometric stability conditions on $X$. Fix $\beta$ and $\omega\in \mathrm{NS}(X)_\Qbb$ with $\omega$ ample. 
\begin{defn}
    For any coherent sheaf $F$, its \emph{slope with respect to $(\beta,\omega)$} is 
    \begin{align*}
        \mu_{\beta,\omega}(F):=\begin{cases}
        \frac{\omega\cdot(c_1(F)-\beta)}{\omega^2} & \mbox{ if } r(F)>0,\\
        +\infty& \mbox{ if } r(F)=0.
        \end{cases}
    \end{align*}
\end{defn}
    We can define an abelian subcategory of $\mathrm{D}^b(X)$ by tilting with respect to
    \begin{align*}
        &\mathcal{T}_{\beta,\omega}:=\{F\in\Coh(X)\;|\;\mbox{all HN factors $F'$ of $F$ have slope }\mu_{\beta,\omega}(F')>0\},\\
        &\mathcal{F}_{\beta,\omega}:=\{F\in\Coh(X)\;|\;\mbox{all HN factors $F'$ of $F$ have slope }\mu_{\beta,\omega}(F')\leqslant 0\}.
    \end{align*}
\begin{prop}\cite{HRS96}\label{tilt}
The category
\begin{align*}
    \mathcal{A}_{\beta,\omega}:=\langle\mathcal{T}_{\beta,\omega}, \mathcal{F}_{\beta,\omega}[1]\rangle
\end{align*}
is an abelian subcategory of $\mathrm{D}^b(X)$.
\end{prop}
Consider the $\C$-linear map 
\begin{eqnarray*}
Z_{\beta,\omega}:K_{\mathrm{num}}(X) & \to & \C, \\
F & \mapsto & (e^{\beta+\sqrt{-1}\omega},\vbf(F)).
\end{eqnarray*}
By \cite[Lemma 6.2]{Bri08}, the pair $\sigma_{\beta,\omega}=(\mathcal{A}_{\beta,\omega},Z_{\beta,\omega})$ defines a \emph{Bridgeland stability condition} if all spherical sheaves $G$ on $X$ satisfy $Z_{\beta,\omega}(G)\notin \R_{\leq0}$ (this condition is satisfied if $\omega^2>2$). From now on, assume $(\beta,\omega)$ is chosen so that $\sigma_{\beta,\omega}$ is a stability condition. We recall
\begin{lem}\cite[Lemma 8.2]{Bri07}\label{grpact}
The group $\mathrm{Aut}(\mathrm{D}^b(X))$ acts on the left on the space of stability conditions $\mathrm{Stab}(X)$ by $\Psi_*(\mathcal{P},Z)=(\Psi(\mathcal{P}), Z\circ\Psi^{-1}_*)$, where $\Psi\in \mathrm{Aut}(\mathrm{D}^b(X))$, $\Psi_*$ also denotes the push forward on the $K$-group, $\mathcal{P}$ is the slicing, and $Z$ is the central charge  of the stability condition. The universal cover $\widetilde{\mathrm{GL}}_2^+(\R)$ of $2\times2$ matrices with real coefficients and positive determinants acts on the right on $\mathrm{Stab}(X)$.
\end{lem}
Any stability condition which is in the $\widetilde{\mathrm{GL}}_2^+(\R)$ orbit of $\sigma_{\beta,\omega}$ is called a \emph{geometric stability condition}. Such stability conditions are characterized by the property that skyscraper sheaves on $X$ are stable. More generally, Bridgeland(\cite[Section 8]{Bri08}) constructed a connected component $\mathrm{Stab}^\dagger(X)$ of the stability manifold $\mathrm{Stab}(X)$ that contains all geometric stability conditions.

Fixing a Mukai vector $\mathbf{v}\in H^*_{\mathrm{alg}}(X,\Z)$, there exists a locally finite set of walls (real codimension one submanifolds with boundary) in $\mathrm{Stab}^\dagger(X)$, determined solely by $\mathbf{v}$, such that the set of $\sigma$-(semi)stable objects does not change within chambers. A stability condition $\sigma$ is called \emph{generic} with respect to $\mathbf{v}$ if it does not lie on a wall for $\mathbf{v}$.

\begin{thm}\cite[Theorem 1.3(a)]{BM14a}
Let $\sigma\in \mathrm{Stab}^\dagger(X)$ be generic with respect to $\mathbf{v}$. There exists a coarse moduli space $M_\sigma(\mathbf{v})$ parametrizing $\sigma$-semistable objects with Mukai vector $\vbf$. Moreover, $M_\sigma(\vbf)$ is a normal irreducible projective variety.\ 

If $\vbf$ is primitive then $M_\sigma(\vbf)=M^{st}_\sigma(\vbf)$ is a projective hyperk\"ahler manifold and the Mukai homomorphism induces an isomorphism
\begin{align*}
    \theta_\vbf:
    \begin{cases}
        \mbox{   }\vbf^\perp\xrightarrow{\sim}\mathrm{NS}(M_\sigma(\vbf))\qquad &\mbox{if }\vbf^2>0,\\
        \vbf^\perp/\vbf\xrightarrow{\sim}\mathrm{NS}(M_\sigma(\vbf))\qquad &\mbox{if }\vbf^2=0.
    \end{cases}
\end{align*}
\end{thm}

We now take a closer look at wall-crossing. Let $\vbf$ be primitive with $\vbf^2\geq -2$. Suppose that $W$ is a wall for $\vbf$, $\sigma_0\in W$ does not lie on any other walls, and $\sigma_+$ and $\sigma_-$ are generic stability conditions near $\sigma_0$ but on opposite sides of $W$. 
\begin{thm}\cite[Theorem 1.4(a)]{BM14a}
The wall $W$ induces birational contractions
\begin{align*}
    \pi^\pm:M_{\sigma_\pm}(\vbf)\to \overline{M}_\pm,
\end{align*}
where $\overline{M}_\pm$ are normal projective varieties. The morphisms $\pi^\pm$ contract curves parametrizing objects that are $S$-equivalent with respect to $\sigma_0$. 
\end{thm}

Walls are classified as follows.
\begin{defn}\cite[Definition 2.20]{BM14b}
    The wall $W$ is called
    \begin{itemize}
    \item a \emph{fake wall} if there are no curves of objects in $M_{\sigma_\pm}(\vbf)$ that are $S$-equivalent objects with respect to $\sigma_0$;
    \item a \emph{flopping wall} if $\overline{M}_+\cong \overline{M}_-$ and the induced birational map $M_{\sigma_+}(\vbf)\dashrightarrow M_{\sigma_-}(\vbf)$ is a flop;
    \item a \emph{divisorial wall} if the morphisms $\pi^\pm$ are both divisorial contractions; in this case $M_{\sigma_+}(\vbf)\cong M_{\sigma_-}(\vbf)$.
\end{itemize}
\end{defn}
General numerical criteria for locating walls were given in \cite[Section 5]{BM14b}. We will describe these for the Hilbert scheme of points in Section~\ref{wallcross}.

We will need the following result.
\begin{prop}\cite[Proposition 2.11]{BM14b}
    The stability conditions $\sigma_{\beta,\omega}$ and $\sigma_{-\beta,\omega}$ are dual to each other: an object $E\in \mathrm{D}^b(X)$ is $\sigma_{\beta,\omega}$-(semi)stable if and only if $\RHom(E,\Ocal_X)[1]$ is $\sigma_{-\beta,\omega}$-(semi)stable.
\end{prop}
\begin{rem}
We note that \cite[Proposition 2.11]{BM14b} uses a shift by $2$ on the dual of the object, whereas we shift by $1$. Our choice does not affect (semi)stability but is more suitable for its application in Section~\ref{wallcross}.
\end{rem}

We recall the comparison between Gieseker and Bridgeland stability when $\omega$ is large.
\begin{thm}\cite[Proposition 14.2]{Bri08}
    Fix $\vbf\in \Halg(X,\Z)$ and $s_0\in\R$. Let $H\in \mathrm{NS}(X)$ be ample with $\mu_{s_0H,H}(\vbf)>0$. Then $M_{\sigma_{s_0H,tH}}(\vbf)=M_H(\vbf)$ for $t\gg0$.
\end{thm}

\section{Brill-Noether loci in $S^{[10]}$ and $M$}\label{BNloci}
We define the Brill-Noether loci in $\hilb$ and $M$ by 
\begin{align*}
   \mathrm{BN}^i(S^{[10]}):=\{\xi\in S^{[10]}\;|\;h^0(\mathcal{I}_\xi(3))\geq i+1\}\subset S^{[10]}
\end{align*}
and
\begin{align*}
    \mathrm{BN}^i(M):=\{\Ecal\in  M\;|\;h^0(\Ecal)\geq i+1\}\subset M.
\end{align*}

\subsection{Brill-Noether loci in $S^{[10]}$}
We will focus on some special subsets in $\BN^i(\hilb)$. We define
\begin{align*}
    W_i:=\{\xi\in \hilb\;|\;\exists\text{ saturated }\xi'\subset\xi \text{ of length $10-i$ and } L\in |H| \text{ such that } \xi'\subset L\}
\end{align*}
for $i=0,1,2,3,4$. Then $W_i$ is a locally closed subset of $\hilb$ with closure
\begin{align*}
    \overline{W_i}=\{\xi\in \hilb\;|\;\exists\text{ }\xi'\subset\xi \text{ of length $10-i$ and } L\in |H| \text{ such that } \xi'\subset L\}.
\end{align*}
Note that $W_0=\overline{W_0}$, $\overline{W_i}\subset \overline{W_{i+1}}$ for $i=0,1,2,3$, and $\overline{W_i}\subset \BN^{5-i}(\hilb)$ for $i=0,1,2,3,4$.

We define
\begin{align*}
    \mathcal{W}_0=\{\xi\in\hilb\;|\;\exists\text{ } Q'\in |2H| \text{ such that } \xi\subset Q'\}
\end{align*}
and
\begin{align*}
    \mathcal{W}_1=\{\xi\in\hilb\;|\;\exists\text{ saturated }\xi'\subset\xi \text{ of length $9$ and } Q''\in |2H| \text{ such that } \xi'\subset Q''\}.
    \end{align*}
Note that $\Wcal_0$ is closed and $\Wcal_1$ is locally closed with closure
\begin{align*}
    \mathcal{W}_1=\{\xi\in\hilb\;|\;\exists\text{ }\xi'\subset\xi \text{ of length $9$ and } Q''\in |2H| \text{ such that } \xi'\subset Q''\}.
    \end{align*}
Moreover, $\mathcal{W}_0\subset\overline{\mathcal{W}_1}$ and $\overline{\Wcal_i}\subset\BN^{2-i}(\hilb)$ for $i=0,1$.

Our next lemma provides bounds for the dimensions of spaces of lines/conics/cubics containing a zero-dimensional closed subscheme.
\begin{lem}\label{linear systems}
Let $\xi\in S^{[10]}$ and let $\xi_i\in S^{[10-i]}$ for $i=0,1,2,3,4,5$.
\begin{enumerate}
    \item We have
    \begin{align*}
        0\leq h^0(S,\Ical_\xi(1))\leq 1\qquad\mbox{and}\qquad 1\leq h^0(S,\Ical_\xi(3))\leq 6.
    \end{align*}
    \item If $h^0(S,\Ical_\xi(1))=1$ then $h^0(S,\Ical_\xi(3))=6$.
    \item 
    If $h^0(S,\Ical_{\xi_i}(1))=1$ then $h^0(S,\Ical_{\xi_i}(2))=3$.
    \item 
    If $h^0(S,\Ical_{\xi}(2))\geq 2$ then $\xi\in\overline{W_2}$.
\end{enumerate}
\end{lem}
\begin{proof}
(1) For $h^0(S,\Ical_\xi)$, we note that any line in $|H|$ is integral. Hence for $L_1,L_2\in |H|$ and $L_1\neq L_2$, $L_1\cap L_2$ is a length $2$ subscheme of $S$. As a result $h^0(S,\Ical_\xi(1))\leq 1$.

The fact that $h^0(S,\Ical_\xi(3))\geq 1$ follows from the short exact sequence
\begin{align*}
    0\to \Ical_\xi(3)\to \Ocal_S(3)\to\Ocal_\xi\to 0.
\end{align*}
Denote by $B(\xi)$ the subspace of $B$ consisting of cubics containing $\xi$. Suppose that $B(\xi)\subset \Sigma$; then $\dim B(\xi)\leq 5$ since there is no 6-dimensional linear subspace of $B=\Pbb^{10}$ in $\Sigma$. Hence we can assume that there is an integral curve $D\in |3H|$ containing $\xi$. When $D$ is smooth, we have the following commutative diagram.
\[
\xymatrix{
0 \ar[r] & 0=H^0(\Ical_\xi) \ar@{^(->}[d] \ar[r] & H^0(\Ical_\xi(3)) \ar@{^(->}[d] \ar@{^(->}[r] \ar[dr]^g & H^0({\Ical_\xi}(3)|_D) \ar[d]^f \ar[r] & \ldots \\
0 \ar[r] & H^0(\Ocal_S)  \ar[r] & H^0(\Ocal_S(3)) \ar[r]  & H^0(\omega_D) \ar[r] & 0
}
\]
We note that $\Ical_\xi(3)|_D=\omega_D(-\xi)\oplus\Ocal_\xi$ and thus $H^0(\Ical(3)|_D)=H^0(\omega_D(-\xi))\oplus H^0(\Ocal_\xi)$, where the second summand is the kernel of $f$. As a result, we have
\begin{align*}
    h^0(\Ical_\xi(3))\leq \dim \mathrm{im}(g)+h^0(\Ocal_S)\leq h^0(D,\omega_D(-\xi))+1=h^0(D,\Ocal_D(\xi))\leq 6,
\end{align*}
where the final inequality follows from Clifford's theorem. The case when $D$ is singular follows from the same argument using an extension of Clifford's theorem to certain singular curves due to Franciosi and Tenni~\cite{FT14} (note that $D$ is automatically $2$-connected because it is integral, and it has locally planar singularities because it is contained in the smooth K3 surface $S$).

\vspace*{2mm}
\noindent(2) Let $0\neq s\in H^0(\Ical_\xi(1))$ and let $L$ be the corresponding line. We have a short exact sequence 
\[  0 \rightarrow \Ocal_S(2) \xrightarrow{s} \Ical_\xi(3) \rightarrow \ker (\Ocal_S(3)|_L \rightarrow \Ocal_\xi) \rightarrow 0.\]
Note that the right term is a torsion free sheaf of rank one on $L$ with negative degree, and thus we have an isomorphism $H^0(\Ical_\xi(3))=H^0(\Ocal_S(2))$.

\vspace*{2mm}
\noindent(3) Let $0\neq s'\in H^0(\Ical_{\xi_i}(1))$ and let $L'$ be the corresponding line. We have a short exact sequence
\[  0 \rightarrow \Ocal_S(1) \xrightarrow{s} \Ical_{\xi_i}(2) \rightarrow \ker (\Ocal_S(2)|_{L'} \to \Ocal_{\xi_i}) \rightarrow 0.\]
Again the right term is a torsion free sheaf of rank one on $L'$ with negative degree, and thus we have an isomorphism $H^0(\Ical_{\xi_i}(2))=H^0(\Ocal_S(1))$.

\vspace*{2mm}
\noindent(4) Suppose $h^0(S,\Ical_{\xi}(2))\geq 2$. Then one can find a morphism $f:\Ocal_S(-2)^{\oplus 2}\to \Ical_\xi$ whose image has $\Ocal_S(-2)$ as a proper subsheaf. Note that both $\ker(f)$ and $\mathrm{im}(f)$ are torsion free sheaves of rank one, hence stable. Combining this with the fact that $\ker(f)\subseteq \Ocal^{\oplus2}_S(-2)$, $\Ocal_S(-2)\subsetneq \mathrm{im}(f)\subseteq \Ical_\xi$, and the stability of all sheaves involved, it is easy to see that $\vbf(\mathrm{im}(f))=(1,-1,2-p)$, where $p$ can be $0,1$, or $2$. If $\mathrm{coker}(f)$ is pure of dimension one then $\xi\in \overline{W_p}$. Otherwise $\xi\in\overline{ W_{p'}}$ for some $p'<p$. In any case, $\xi\in \overline{W_2}$.
\end{proof}

As a result of part (1) of the lemma, we see that $\BN^0(\hilb)=\hilb$ and $\BN^6(\hilb)=\emptyset$. 

In the next three propositions, we investigate the structure of the $W_i$'s and $\Wcal_j$'s. Note that the universal sheaves appearing below exist by \cite[Corollary 4.6.7]{HL10}. We first show that all $W_i$ are generically $\Pbb^{8-i}$-bundles.
\begin{prop}\label{part1}
    \begin{enumerate}[(i)]
    \item The variety $W_0$ is a $\Pbb^8$-bundle over $M(0,1,-11)$. More precisely, let $\mathscr{U}_{-11}$ be the universal sheaf on $M(0,1,-11)\times S$ and define $$\mathscr{E}_0:=p_{1*}R\mathcal{H}om(\mathscr{U}_{-11},p_2^*\Ocal_S(-1))[1].$$
    Then $\mathscr{E}_0$ is a vector bundle of rank $9$ on $M(0,1,-11)$ and $W_0\cong\Pbb(\mathscr{E}_0)$.
        \item For $i=1,2,3$ the variety $W_{i}\backslash \overline{W_{i-1}}$ is a $\Pbb^{8-i}$-bundle over an open subset of $S^{[i]}\times M(0,1,i-11)$.
    \end{enumerate}
\end{prop}
\begin{proof}
(i) The proof of the first part is same as the proof of \cite[Proposition 4.4]{Hel20}.

\vspace*{2mm}
\noindent(ii) Fix $i\in\{1,2,3\}$. Let $V_{i}\subset S^{[i]}\times|H|$ be the open set parametrizing pairs $(\xi, L)$ such that $\xi\cap L=\emptyset$. Define $U_{i}$ to be the preimage of $V_{i}$ by the support map $S^{[i]}\times M(0,1,i-11)\to S^{[i]}\times|H|$. For $\xi\in W_i\backslash\overline{W_{i-1}}$, there exists a saturated subscheme $\zeta$ of length $i$ so that $\xi\backslash\zeta\subset L$ for some $L\in |H|$. We claim that both $\zeta$ and $L$ are unique. Suppose $\zeta'\subset \xi$ satisfies $l(\zeta')=i$ and $\xi\backslash\zeta'\subset L'$ for $L'\in |H|$. Then $\xi\backslash(\zeta\cup\zeta')$ is a subscheme of $\xi$ so that $l(\xi\backslash(\zeta\cup\zeta'))\geq 10-2i\geq4$. On the other hand, $\xi\backslash(\zeta\cup\zeta')\subset L\cap L'$. If $L\neq L'$ then $L\cap L'$ is a length two subscheme of $S$. This contradiction proves that $L=L'$.  Next note that $\zeta\backslash\zeta'\subset \xi\backslash\zeta'\subset L'=L$, so if $\zeta\backslash\zeta'\neq \emptyset$ then $(\xi\backslash\zeta) \cup (\zeta\backslash\zeta')\subset L$ contradicts our assumption that $\xi\notin \overline{W_{i-1}}$. This proves the claim. As a result, we get a well-defined morphism
\begin{eqnarray*}
\psi_i: W_i\backslash\overline{W_{i-1}} & \to & S^{[i]}\times M(0,1,i-11), \\
\xi & \mapsto & (\zeta,\Ical_{(\xi\backslash\zeta)/L}).
\end{eqnarray*}
It is clear that the image is contained in $U_i$. We now show that $\psi_i$ is surjective onto $U_i$ and that the fiber over each point is isomorphic to $\Pbb^{8-i}$. We first note that if 
$\psi_i(\xi)=(\zeta,\Ical_{(\xi\backslash\zeta)/L})$ then we have a short exact sequence
\begin{align*}
    0\to \Ical_\zeta(-1)\to\Ical_\xi\to\Ical_{(\xi\backslash\zeta)/L}\to 0.
\end{align*}
On the other hand, it is easy to see that for any $(\eta,\Ecal)\in U_i$, $\Ext^1(\Ecal,\Ical_\eta(-1))\cong\C^{9-i}$. For any non-split extension
\begin{align*}
    0\to \Ical_\eta(-1)\to I \to\Ecal\to 0,
\end{align*}
$I$ must be a torsion-free sheaf with Mukai vector $(1,0,-9)$, since $\eta\cap \Supp(\Ecal)=\emptyset$. Hence $I$ is the ideal sheaf of a length $10$ subscheme $\xi'$, with $\eta\subset \xi'$ as a saturated subscheme. It is clear that $\xi'\in W_i$. If $\xi'\in \overline{W_{i-1}}$ then there would exist a subscheme $\eta'\subset\xi'$ of length $11-i$ such that $\eta'$ lies on a line. Consider $\eta'\backslash\eta$, which is a closed subscheme of $\xi'\backslash \eta\subset \Supp(\Ecal)$ and has length $\geq 11-2i\geq 5$. Since any subscheme of length $3$ or more is contained in at most one line, we must have $\eta'\subset \Supp(\Ecal)$. This contradicts the fact that $\eta\cap \Supp(\Ecal)=\emptyset$, as $l(\xi'\backslash\eta)=10-i$. Hence $\xi'\in W_i\backslash\overline{W_{i-1}}$. It follows that $\psi_i$ is surjective onto $U_i$ and the fiber over any $(\eta,\Ecal)\in U_i$ is $\Pbb(\Ext^1(\Ecal,\Ical_\eta(-1)))\cong \Pbb^{8-i}$.
\end{proof}

\begin{rem}
Similarly to part (i), the $\Pbb^6$-bundle structure of  $W_2\backslash\overline{W_1}$ over $U_2$ can be realized as the projectivization of a vector bundle on $U_2$. More precisely, let $\mathscr{U}_{-9}$ be the universal sheaf on $M(0,1,-9)\times S$ and $\mathscr{I}_{2}$ be the universal ideal sheaf on $S^{[2]}\times S$. Define
$$\mathscr{E}_{2}:=p_{12*}R\mathcal{H}om(p^*_{23}\mathscr{U}_{-9},p_3^*\Ocal_S(-1)\otimes p^*_{13}\mathscr{I}_{2})[1],$$
where $p_{12}$, $p_{13}$, and $p_{23}$ are the obvious projections from $S^{[2]}\times M(0,1,-9)\times S$. Then one can easily adapt the proof above to show that $\mathscr{E}_{2}$ is a vector bundle of rank $7$ on $U_2$ and $W_2\backslash\overline{W_{1}}\cong\Pbb(\mathscr{E}_{2}|_{U_{2}})$.

For $i=1$ and $3$ we only have twisted universal sheaves on $M(0,1,-11+i)\times S$. In these cases $W_i\backslash \overline{W_{i-1}}$ can be viewed as the projectivization of a twisted vector bundle on $U_i$.
\end{rem}

Next we describe $W_4$ and $\Wcal_0$. The reason for describing them together is that their strict transforms are the two irreducible components of an exceptional locus in Theorem \ref{main}.
\begin{prop}\label{part2}
    \begin{enumerate}[(i)]
         \item The variety $W_4\backslash(\Wcal_0\cup \overline{W_3})$ is a $\Pbb^4$-bundle over an open subset in $S^{[4]}\times M(0,1,-7)$. More precisely,  let $\Usc_{-7}$ be the universal sheaf on $M(0,1,-7)\times S$ and $\Isc_{4}$ the universal ideal sheaf on $S^{[4]}\times S$. Define $$\Esc_{4}:=p_{12*}R\mathcal{H}om(p^*_{23}\Usc_{-7},p_3^*\Ocal_S(-1)\otimes p^*_{13}\Isc_{4})[1].$$
        Then $\Esc_{4}$ is a vector bundle of rank $5$ on an open subset $U_{4}\subset S^{[4]}\times M(0,1,-7)$ and $W_{4}\backslash (\Wcal_0\cup \overline{W_3})=\Pbb(\Esc_{4}|_{U_4})$.\ 
        \item The variety $\Wcal_0\backslash \overline{W_4}$ is generically a $\Pbb^5$-bundle over $M^{st}(0,2,-14)$.
        \item Let $N$ be the moduli space parametrizing pairs $(\xi_4,L)$ such that $\xi_4\subset L\in|H|$ is a subscheme of length $4$. Then there is a $\Pbb^4$-bundle $\Pbb(\Esc_5|_{U_5})$ over an open subset $U_5\subset N\times M(0,1,-7)$ and a generically injective morphism from $\Pbb(\Esc_5|_{U_5})$ onto $(\Wcal_0\cap W_4)\backslash \overline{W_3}$.
    \end{enumerate}
\end{prop}
\begin{proof}
\noindent(i)  Let $V_4\subseteq S^{[4]}\times |H|$ be the open set parametrizing pairs $(\xi,L)$ such that $\xi\cap L=\emptyset$ and $H^0(\Ical_\xi(1))=0$. Define $U_4$ as the preimage of $V_4$ by the support map $S^{[4]}\times M(0,1,-7)\to S^{[4]}\times |H|$. Then one proceeds as before to show that  $\Esc_4$ is a vector bundle of rank $5$ on $U_4$ and that the universal extension defines a map
\begin{align*}
    \psi_4:\Pbb(\Esc_4|_{U_4})\to \hilb
\end{align*}
whose image lies in $W_4$.

We claim that $\psi_4$ is an isomorphism onto $W_4\backslash(\Wcal_0\cup \overline{W_3})$. First we show that $\mathrm{im}(\psi_4)=W_4\backslash(\Wcal_0\cup \overline{W_3})$. Assume $\xi\in \mathrm{im}(\psi_4)$. Then $\Ical_\xi$ fits into a short exact sequence
\begin{align*}
    0\to \Ical_{\xi_4}(-1)\to \Ical_\xi\to \Ecal\to 0
\end{align*}
where $\xi_4\subset \xi$ is a subscheme of length $4$ and $\Supp(\xi_4)\cap\Supp(\Ecal)=\emptyset$ (hence $\xi_4\subset\xi$ is saturated). As in the proof of Proposition \ref{part1}(ii), $\xi\notin \overline{W_3}$. Suppose $\xi\in\Wcal_0$. By Lemma \ref{linear systems}(iii), any curve $Q\in|2H|$ containing $\xi\backslash\xi_4$ must be of the form $L_1+L_2$ where $L_i\in |H|$ and $L_1=\Supp(\Ecal)$. As a result, $\xi_4$ must be a subscheme of $L_2$. This contradicts our choice of $V_4$. We have shown that $\mathrm{im}(\psi_4)\subseteq W_4\backslash(\Wcal_0\cup \overline{W_3})$. Next suppose that $\xi\in W_4\backslash(\Wcal_0\cup \overline{W_3})$. Since $\xi\in W_4\backslash\overline{W_3}$, we have a saturated subscheme $\xi_4\subset \xi$ such that $\Ical_\xi$ fits into
\begin{align*}
    0\to \Ical_{\xi_4}(-1)\to \Ical_\xi\to \Ecal\to 0
\end{align*}
where $\Ecal$ is pure of dimension one (hence stable) and $\Supp(\xi_4)\cap\Supp(\Ecal)=\emptyset$. It remains to show that $H^0(\Ical_{\xi_4}(1))=0$. Suppose otherwise; then the composition of any nonzero (hence injective) $\Ocal_S(-2)\to \Ical_{\xi_4}(-1)$ with $\Ical_{\xi_4}(-1)\to \Ical_\xi$ implies $\xi\in\Wcal_0$.

Secondly, we note that $\Hom(\Ical_{\xi_4}(-1),\Ecal)=0$ for any $(\xi_4,\Ecal)\in U_4$, so $\psi_4$ is a local isomorphism. Lastly, we show that $\psi_4$ is injective. Suppose that $\Ical_\zeta$ with $\zeta\in W_4\backslash(\Wcal_0\cup \overline{W_3})$ can be written as two different extensions, over $(\xi,\Ecal)$ and $(\xi',\Ecal')$ in $U_4$. By the definition of $U_4$, we can use $\zeta\backslash(\xi\cup\xi')$ to denote the subscheme of $\zeta$ supported on the complement of the union of the support of $\xi$ and $\xi'$, and then $l(\zeta\backslash(\xi\cup\xi'))\geq2$. If $h^0(\Ical_{\zeta\backslash(\xi\cup\xi')}(1))=1$ (for example, this happen when $l(\zeta\backslash(\xi\cup\xi'))\geq3$) then one can use the argument from the proof of Proposition \ref{part1}(ii) to conclude injectivity. Suppose $h^0(\Ical_{\zeta\backslash(\xi\cup\xi')}(1))\geq2$ and $\Supp(\Ecal)\neq\Supp(\Ecal')$. Then $\zeta=\xi\cup\xi'\cup (\Supp(\Ecal)\cap\Supp(\Ecal'))$ and $\zeta$ lies in a reducible curve in $|2H|$ whose support is $\Supp(\Ecal)\cup\Supp(\Ecal')$. Hence $\zeta\in \Wcal_0$, contradicting our assumption on $\zeta$.

\vspace*{2mm}
\noindent(ii) Let $\xi\in\Wcal_0\backslash \overline{W_4}$. By Lemma \ref{linear systems}(4), there exists a unique (up to a scalar) injection $\Ocal_S(-2)\to \Ical_\xi$. Then $\Ical_\xi$ fits into an exact sequence
\begin{align*}
    0\to\Ocal_S(-2)\to\Ical_\xi\to \Ecal\to 0.
\end{align*}
Note that $\Ecal$ must be pure of dimension one by the Bogomolov inequality. If $\Ecal$ is not stable, let $\Ecal''$ be its maximal destabilizing quotient. Then $\vbf(\Ecal'')=(0,1,-11+i)$ where $i\in\{0,1,2,3,4\}$. This would imply that $\xi\in \overline{W_4}$. Hence $\Ecal\in M^{st}(0,2,-14)$ and we have a morphism
\begin{align*}
   \phi: \Wcal_0\backslash\overline{W_4}\to  M^{st}(0,2,-14).
\end{align*}
For any $\Ecal\in M^{st}(0,2,-14)$, we have $\Ext^1(\Ecal,\Ocal(-2))\cong\C^6$. If a non-split extension of $\Ecal$ and $\Ocal(-2)$ is torsion free, then it is the ideal sheaf of a length $10$ subscheme $\xi'$ in $S$. If $\xi'\in\overline{W_4}$ then there exists $L\in|H|$ and $i\in\{0,1,2,3,4\}$ such that we have 
\begin{align}\label{eq2}
    0\to \Ical_{\zeta'}(-1)\to \Ical_{\xi'}\to \Ical_{\xi'\cap L/L}\to 0,
\end{align}
where $\xi'\cap L$ ($\xi'\times_SL$ to be precise) is a subscheme of length $10-i$ of both $\xi'$ and $L$, and $\zeta'$ is of length $i$. By Lemma \ref{linear systems}(3), the map $\Ocal_S(-2)\to \Ical_{\xi'}$ factors through the morphism $\Ical_{\zeta'}(-1)\to \Ical_{\xi'}$ of (\ref{eq2}). Hence $\Ecal$ has $\Ical_{\xi'\cap L/L}$ as a quotient, which contradicts the fact that $\Ecal$ is stable and $\vbf(\Ical_{\xi\cap L/L})=(0,1,-11+i)$. As a result, we see that $\xi'\in \Wcal_0\backslash \overline{W_4}$.

Suppose $T$ is a non-split extension of $\Ecal$ and $\Ocal_S(-2)$,
\begin{align*}
    0\to \Ocal_S(-2)\to T\to \Ecal\to 0,
\end{align*}
with nontrivial torsion part. Then $T_{tor}\subset \Ecal$ must have first Chern class $H$. Since $T/T_{tor}$ is torsion free, $\Ecal/T_{tor}$ is pure of dimension one.  Alternatively, we can write $T$ as non-split extensions
$$0\to (\Ocal_S(-2)\oplus T_{tor})\to T\to \Ecal/T_{tor}\to 0$$
and
$$0\to T_{tor}\to T \to T_{tf}\to 0,$$
where $T_{tf}$ is the torsion free quotient of $T$, which is a non-split extension of $\Ecal/T_{tor}$ by $\Ocal_S(-2)$. By the stability of $\Ecal$ and $T_{tf}$, we see that $T_{tor}\in M(0,1,j)$ where $-11\leq j\leq -8$. We note that such an extension with nontrivial torsion part exists only when $\Supp(\Ecal)$ is not integral. Let $M^{st}_0(0,2,-14)$ denote the open subset in $M^{st}(0,2,-14)$ parametrizing sheaves with integral support. Then the above shows that  $\phi|_{\phi^{-1}(M^{st}_0(0,2,-14))}$ is surjective onto $M^{st}_0(0,2,-14)$ with fiber over each $\Ecal$ being isomorphic to $\Pbb(\Ext^1(\Ecal,\Ocal(-2)))\cong\Pbb^5$. This proves that $\Wcal_0\backslash\overline{W_4}$ is generically a $\Pbb^5$-bundle over $M^{st}(0,2,-14)$.  Note that over the complement of $M^{st}_0(0,2,-14)$, $\phi$ is surjective with fiber  over $\Ecal'$ being a subset of $\Pbb(\Ext^1(\Ecal',\Ocal(-2)))\cong\Pbb^5$.

\vspace*{2mm}
\noindent(iii) We describe $N$ first. Let $\Usc_{-7}$ be the universal sheaf on $M(0,1,-7)\times S$, and define
$$\mathcal{N}:=p_{1*}R\mathcal{H}om(\Usc_{-7},p_2^*\Ocal_S(-2))[1].$$
It is easy to see that $\mathcal{N}$ is a vector bundle of rank $3$ on $M(0,1,-7)$ and $N=\Pbb(\mathcal{N})$ parametrizes non-split extensions of the form
\begin{align*}
    0\to \Ocal_S(-2)\to \Ical\to \Ecal\to 0,
\end{align*}
where $\Ecal\in M(0,1,-7)$. As before, we see that $\Ical$ is torsion free, hence of the form $\Ical_{\xi_4}(-1)$ where $\xi_4\subset \Supp(\Ecal)$ is a subscheme of length $4$. Moreover, the natural morphism $N\to S^{[4]}$ is a closed embedding and $N$ parametrizes length $4$ subscheme of $S$ lying on a line $L\in |H|$. Note that if $\Isc_4$ is the universal family on $S^{[4]}\times S$, then $(\Isc_4\otimes p_2^*\Ocal_S(-1))|_{N\times S}$ is a universal family parametrizing $\Ical$ of the form above.\ 

Let $U_5\subset N\times M(0,1,-7)$ be the open subscheme parametrizing pairs $(\xi_4,\Ecal')$, where $\xi_4\in N$ and $\Ecal'\in M(0,1,-7)$, such that  $\xi_4\cap\Supp(\Ecal')=\emptyset$. Define 
$$\Esc_5:=p_{12*}R\mathcal{H}om(p^*_{23}\Usc_{-7},p^*_{13}((\Isc_4\otimes p_2^*\Ocal_S(-1))|_{N\times S}))[1].$$
It is easy to see that $\Esc_5$ is a vector bundle of rank $5$ on $U_5$ and $\Pbb(\Esc_5|_{U_5})$ parametrizes non-split extensions of  $\Ecal'$ with $\Ical_{\xi_4}(-1)$ for $(\xi_4,\Ecal')\in U_5$. As in part (ii), we can show that a non-split extension
\begin{align*}
    0\to \Ical_{\xi_4}(-1)\to \Ical\to \Ecal'\to 0
\end{align*}
is torsion free, and hence $\Ical=\Ical_\zeta$ for some $\zeta\in\hilb$. Moreover, we see that $\zeta\backslash\xi_4\subset\Supp(\Ecal')$ is a saturated subscheme of $\zeta$, so $\zeta\in W_4$. Since there is an inclusion $\Ocal(-2)\to \Ical_\zeta$, we also have $\zeta\in \Wcal_0$. By our choice of $U_5$, we know that $\zeta\notin \overline{W_3}$. So the universal extension defines a map
\begin{align*}
    \psi_5: \Pbb(\Esc_5|_{U_5})\to \hilb
\end{align*}
whose image is in $(W_4\cap \Wcal_0)\backslash\overline{W_3}$. Since $\Hom(\Ical_{\xi_4}(-1),\Ecal')=0$ for $(\xi_4,\Ecal')\in U_5$, $\psi_5$ is a local isomorphism. We claim that $\psi_5$ maps onto $(W_4\cap \Wcal_0)\backslash\overline{W_3}$. Suppose that $\zeta\in (W_4\cap \Wcal_0)\backslash\overline{W_3}$; then there exists a saturated subscheme $\zeta'\subset \zeta$ of length $6$ contained in a line $L_1\in |H|$. By Lemma \ref{linear systems}(iii), any conic containing $\zeta$ (hence containing $\zeta'$) must be of the form $L_1+L_2$, where $L_2\in |H|$ with $L_2\neq L_1$. As $L_2$ must contain the length $4$ subscheme $\zeta\backslash\zeta'$, the conic $L_1+L_2$ is uniquely determined by the fact that it contains $\zeta$. As $\zeta\notin \overline{W_3}$, we have $(\zeta\backslash\zeta')\cap \Supp(L_1)=\emptyset$.  There exists an inclusion 
\begin{align*}
    \Ical_{\zeta\backslash\zeta'}(-1)\hookrightarrow \Ical_\zeta
\end{align*}
whose cokernel $\Ecal'$ is the ideal sheaf of $\zeta'$ in $L_1$. Thus $\Ecal'$ is pure of dimension one with Chern character $(0,1,-7)$, and hence stable. To see that $\psi_5$ is generically injective, we simply notice that $\psi_5$ fails to be injective on $\Ical_\zeta$ only when $\zeta=\zeta_1\cup\zeta_2\cup\zeta_3$, where $\zeta_1$ and $\zeta_2$ are both saturated of length $4$ and contained in lines $L_1$ and $L_2$, respectively, and $\zeta_3=L_1\cap L_2$.
\end{proof}

Lastly, we give a description of $\Wcal_1$. We will need a birational model $X_5$ of $\hilb$ defined in Section~4. It is obtained from $\hilb$ by a sequence of flops at the $W_i$'s and $\Wcal_0$ (see Theorem \ref{main} for details). 
\begin{prop}\label{part3}
    Let $\mathscr{V}'$ be the universal sheaf on $M(0,2,-13)\times S$ and let $\Isc_{\Delta}$ be the ideal sheaf of the diagonal on $S\times S$. Define
    $$\Fsc':=p_{12*}R\mathcal{H}om(p^*_{23}\mathscr{V}',p_3^*\Ocal_S(-2)\otimes p^*_{13}\Isc_{\Delta})[1].$$
    Then $\Fsc'$ is a vector bundle of rank $5$ on an open subset $U'\subset S\times M(0,2,-13)$. Moreover, there exists a generically injective morphism $\Pbb(\Fsc'|_{U'})\to X_5$ whose image contains $\Wcal_1\backslash(\Wcal_0\cup \overline{W_4})$ as an open subscheme.
\end{prop}
\begin{proof}
Let $V'\subset S\times |2H|$ be the open set parametrizing pairs $(p,Q)$ such that $p\notin Q$. Define $U'$ as the preimage of $V'$ by the support map $S\times M(0,2,-13)\to S\times |2H|$. It is easy to see that if $(p,\Ecal)\in U'$ then
\begin{align*}
    R\Hom_S(\Ecal,\Ical_p(-2))[1]\cong\Ext^1(\Ecal,\Ical_p(-2))\cong\C^{5}.
\end{align*}
We conclude that $\Fsc'$ is a vector bundle of rank $4$. If a non-split extension of $\Ecal$ and $\Ical_p(-2)$ is torsion free then it must be an ideal sheaf of a length $10$ subscheme $\xi$ in $S$. It is clear that $\xi\in\Wcal_1\backslash\Wcal_0$. If $\xi\in \overline{W_4}$ then there exists $L\in |H|$ and $i\in\{0,1,2,3,4\}$ such that we have 
\begin{align}\label{eq3}
    0\to \Ical_{\xi'}(-1)\to \Ical_\xi\to \Ical_{\xi\cap L/L}\to 0,
\end{align}
where $\xi\cap L$ ($\xi\times_SL$ to be precise) is a subscheme of length $10-i$ of both $\xi$ and $L$, and $\xi'\subset\xi$ is a subscheme of length $i$. Now as $(\xi\cap L)\backslash \{p\}$ has length at least $5$, Lemma \ref{linear systems}(3) implies that $Q=\Supp(\Ecal)$ is of the form $L+L'$ for $L'\in|H|$. Note that $p\notin L\cup L'$ since $\xi\notin \Wcal_0$, and $\xi'\subset L'\cup \{p\}$. We see now that the map $\Ical_p(-2)\to \Ical_\xi$ factors through the morphism $\Ical_{\xi'}(-1)\to \Ical_\xi$ of (\ref{eq3}). Hence $\Ecal$ has $\Ical_{\xi\cap L/L}$ as a quotient, which contradicts the fact that $\Ecal$ is stable as $\vbf(\Ical_{\xi\cap L/L})=(0,1,-11+i)$. As a result, we have $\xi\in\Wcal_1\backslash(\Wcal_0\cup \overline{W_4})$.

Suppose $T$ is a non-split extension of $\Ecal$ and $\Ical_p(-2)$ with nontrivial torsion part. We argue as in the proof of Proposition \ref{part2}(ii) to show that $T$ is of the form
\begin{align*}
    0\to T_{tor}\to T\to T_{tf}\to 0,
\end{align*}
where $T_{tor}\in M(0,1,j)$ where $-11\leq j\leq -7$. Note that such $T$ lies in the exceptional locus in $X_5$ for $g_4$ by the analysis of the fifth wall in the proof of Theorem \ref{main}.

Combining the previous two paragraphs, the associated universal extension defines a map 
\begin{align*}
    \phi':\Pbb(\Fsc'|_{U'})\to X_5.
\end{align*}
Since $\Hom(\Ical_p(-2),\Ecal)=0$ for $(p,\Ecal)\in U'$, $\phi'$ is a local isomorphism. Now suppose that $\xi \in \Wcal_1\backslash(\Wcal_0\cap\overline{W_4})$. Then there exists a conic $Q\in |2H|$ such that $\xi \cap Q$ is a saturated subscheme of $\xi$ of length $9$. There is a surjection $\Ical_\xi\to \Ecal':=\Ical_{Q\cap\xi/Q}$ whose kernel is $\Ical_p(-2)$ for some point $p\notin Q$. Now that we have short exact sequence
\begin{align*}
    0\to \Ical_p(-2)\to \Ical_\xi\to \Ecal'\to 0,
\end{align*}
it remains to show that $\Ecal'$ is stable with Mukai vector $(0,2,-13)$. Note that if $\Ecal'$ is not pure then $\xi$ would be contained in a conic by the Bogomolov inequality, which contradicts our assumption. Suppose $\Ecal'$ is not stable and let $\Ecal''$ be the maximal destabilizing quotient. Then $\vbf(\Ecal'')=(0,1,-11+i)$ where $i\in\{0,1,2,3,4\}$. This would imply that $\xi\in \overline{W_4}$, violating our assumption. So $\Wcal_1\backslash(\Wcal_0\cap\overline{W_4})$ is a subset of $\mathrm{im}(\phi')$. Moreover, $\Wcal_1\backslash(\Wcal_0\cap\overline{W_4})$ is the complement of the exceptional locus mentioned in the previous paragraph in $\mathrm{im}(\phi')$, and thus $\Wcal_1\backslash(\Wcal_0\cap\overline{W_4})$ is open in $\mathrm{im}(\phi')$.

To see that $\phi'$ is generically injective define 
\begin{align*}
    \mathscr{W}=\{\xi\in\Wcal_1\backslash(\Wcal_0\cap\overline{W_4})\;|\text{ there is no length $8$ subscheme }\xi'\subset \xi \text{ with } H^0(\Ical_{\xi'}(2))\geq 2\}.
\end{align*}
Then $\mathscr{W}$ is open in $\Wcal_1\backslash(\Wcal_0\cap\overline{W_4})$, and hence open in $\mathrm{im}(\phi')$, which is irreducible. Suppose that $\xi\in \mathscr{W}$ can be expressed as an extension in two different ways, corresponding to elements of $\Ext^1(\Ecal, \Ical_p(-2))$ and $\Ext^1(\Ecal',\Ical_{p'}(-2))$. Then $\xi\backslash (p\cup p')$ is a length $8$ saturated subscheme of $\xi$ contained in both $\Supp(\Ecal)$ and $\Supp(\Ecal')$. It follows from the definition of $\mathscr{W}$ that $\Supp(\Ecal)=\Supp(\Ecal')$, and hence $p=p'$ and $\Ecal=\Ecal'$.
\end{proof}
\begin{rem}
We note that in Propositions \ref{part1}, \ref{part2} and \ref{part3}, the opens sets over which the projective bundles are defined are the largest open sets with the given properties.
\end{rem}
\begin{rem}
The dimension of $W_i$ is $12+i$ for $i\in\{0,1,2,3,4\}$. The dimension of $\Wcal_i$ is $15+i$ for $i\in\{0,1\}$. The dimension of $W_4\cap \Wcal_0$ is $14$.
\end{rem}

\subsection{Brill-Noether loci in $M$} Let $B^\circ\subset B$ be the locus of smooth curves in $B=|3H|$ and let $\mathcal{C}^\circ\to B^\circ$ be the universal curve. Let $M(0,3,k-9)^\circ$ be the preimage of $B^\circ$ for the support map $M(0,3,k-9)\to B$. Recall that we defined
\begin{align*}
    \mathrm{BN}^i(M):=\{\Ecal\in M\;|\;h^0(\Ecal)\geq i+1\}\subset M.
\end{align*}
Now we also define
\begin{align*}
    \mathrm{BN}^i_k(B^\circ):=\{\mathcal{L}\in M(0,3,k-9)^\circ\;|\;h^0(\mathcal{L})\geq i+1\}\subset M(0,3,k-9)^\circ
\end{align*}
and 
\begin{align*}
    Z_{1,3}^\circ:=\{\Lcal\in M(0,3,-7)^\circ\;|\;\Lcal\cong\Ocal_C(p_1+p_2+p_3-p_4) \text{ for some }C\in B^\circ\text{ and }p_i\in C\}.
\end{align*}
We are interested in the following sets:
\begin{align*}
    &Z_2:=\overline{\mathrm{BN}^0_2(B^\circ)}\subset M(0,3,-7),\\
    &Z_4:=\overline{\mathrm{BN}^0_4(B^\circ)}\subset M(0,3,-5),\\
    &Z_{1,3}:=\overline{Z_{1,3}^\circ}\subset M(0,3,-7),\\
    &Z_8:=\overline{\mathrm{BN}^0_8(B^\circ)}\subset M.
\end{align*}
We will treat $Z_2$, $Z_4$, and $Z_{1,3}$ as subschemes of $M$ via the following isomorphisms:
\begin{align*}
    M(0,3,-7)\to M,&\qquad \mathcal{E}\mapsto \mathcal{E}\otimes\Ocal_S(1),\\
    M(0,3,-5)\to M,&\qquad \mathcal{E}\mapsto \mathcal{E}xt^1(\mathcal{E},\Ocal_S(-1)).
\end{align*}
We also obtain an isomorphism 
\begin{align*}
    M(0,3,-7)\to M(0,3,-5),&\qquad \mathcal{E}\mapsto \mathcal{E}xt^1(\mathcal{E},\Ocal_S(-2))
\end{align*}
by composing the first morphism above with the inverse of the second.

\begin{lem}
We have $Z_2\subset \BN^2(M)$ and $Z_4$, $Z_{1,3}\subset \BN^1(M)$. We have $Z_2\subset Z_4\subset Z_8$ and $Z_2\subset Z_{1,3}\subset Z_8$. 
\end{lem}
\begin{proof}
It suffices to prove the lemma over $C\in B^\circ$. Let $\Lcal_2\in \Pic^2(C)$ be a line bundle on $C$ with $H^0(C,\Lcal_2)\neq 0$. We have an injective morphism
\begin{align*}
     \Ocal_C(1)\hookrightarrow\Ocal_C(1)\otimes\Lcal_2.
\end{align*}
Then $h^0(C,\Ocal_C(1)\otimes \Lcal_2)\geq 3$ because $h^0(C,\Ocal_C(1))=h^0(S,\Ocal_S(1))=3$. Thus $Z_2\subset \BN^2(M)$.

Let $\Lcal_4\in\Pic^4(C)$ be a line bundle on $C$ with $H^0(C,\Lcal_4)\neq 0$. Let $\xi_4\subset C$ be the zero divisor of a section of $\Lcal_4$. We have a short exact sequence
\begin{align*}
    0\to \Ocal_S(-1)\to\Ical_{\xi_4}(2)\to \Lcal_4^\vee(2)\to 0.
\end{align*}
Then $h^0(C,\Lcal_4^\vee(2))\geq h^0(S,\Ical_{\xi_4}(2))\geq 2$. Thus $Z_4\subset \BN^1(M)$.
\ 

Let $\Lcal_{1,3}=\Ocal_C(p_1+p_2+p_3-p_4)\in \Pic^2(C)$ for $p_i\in C$. Then by \cite[Lemma IV.5.5]{Har77},
\begin{align*}
h^0(\Lcal_{1,3}\otimes \Ocal_C(1))\geq h^0(\Ocal_C(p_1+p_2+p_3))+h^0(\Ocal_C(1)\otimes\Ocal_C(-p_4))-1\geq 1+2-1=2.
\end{align*}
Thus $Z_{1,3}\subset \BN^1(M)$.

Next we compare the $Z_i$'s. To show that $Z_2\subset Z_4$ we need to show that for $\Lcal_2$ as above, $h^0(C,\Lcal_2^\vee(1))=h^1(C,\Lcal_2(2))\neq 0$. Note that $\chi(\Lcal_2(2))=5$. We can argue as in the first paragraph to show that $h^0(C,\Lcal_2(2))\geq h^0(C,\Ocal_C(2))=6$. Thus $h^1(C,\Lcal_2(2))\neq0$. To show that $Z_4\subset Z_8$ we need to show that with $\Lcal_4$ as above, $h^0(C,\Lcal_4^\vee(2))\geq 1$. This has been proved above. Finally, it is clear that $Z_2\subset Z_{1,3}\subset Z_8$.
\end{proof}

\begin{prop}\label{BNforM}
    \begin{enumerate}
        \item $Z_2$ is a $\Pbb^8$-bundle over $S^{[2]}$. In particular, $Z_2$ is smooth of dimension $12$.
        \item $Z_4$ is generically isomorphic to a $\Pbb^6$-bundle over $S^{[4]}$.
        \item $Z_{1,3}$ is generically isomorphic to a $\Pbb^6$-bundle over $S^{[1]}\times S^{[3]}$.
        \item $Z_8$ is generically isomorphic to a $\Pbb^2$-bundle over $S^{[8]}$.
    \end{enumerate}
\end{prop}
\begin{proof}
    First some preliminaries. For $i\geq 1$ let $\mathcal{Z}_i\subset S\times S^{[i]}$ be the universal subscheme and let $\Isc_{\Zcal_i}$ be the ideal sheaf of $\Zcal_i$. Let $p_1$ and $p_2$ be the projections from $S\times S^{[i]}$ to its two components. We have an inclusion 
    \begin{align*}
        p_{2*}(\Isc_{\Zcal_i}\otimes p_1^*\Ocal_S(3))\hookrightarrow p_{2*}(\Ocal_{S\times S^{[i]}}\otimes p_1^*\Ocal_S(3))\cong H^0(S,\Ocal_S(3))\otimes\Ocal_{S^{[i]}}.
    \end{align*}
    This defines a subscheme $\Xcal_i\subset B\times S^{[i]}$, which can be thought of as $\mathrm{Hilb}^i(\mathcal{C}/B)$. We note that $\Xcal_i$ is irreducible. Similarly, one can construct $\Xcal_{1,3}\subset B\times S\times S^{[3]}$, parametrizing triples $(C,p,\xi_3)$ where $C\in B$, $p\in C$, and $\xi_3\subset C$.
    
    We claim that $\Xcal_2$ is a $\Pbb^8$-bundle over $S^{[2]}$. We need to show that for any $\xi_2\in S^{[2]}$, $h^0(\Ical_{\xi_2}(3))=9$. We have a long exact sequence
    \begin{align*}
        0\to H^0(\Ical_{\xi_2}(3))\to H^0(\Ocal_S(3))\xrightarrow{r} H^0(\Ocal_{\xi_2})\to H^1(\Ical_{\xi_2}(3))\to0.
    \end{align*}
    Since $h^0(\Ocal_S(3))=11$ and $h^0(\Ocal_{\xi_2})=2$, it remains to show that the map $r$ is surjective. Recall that $H^0(\Ocal_S(3))$ has the pull-backs of  cubics $H^0(\Pbb^2,\Ocal_{\Pbb^2}(3))$ as a $10$-dimensional subspace, along with an extra dimension generated by the reduced ramification locus $R$. If $\xi_2=p\sqcup q$ with $\pi(p)\neq \pi(q)$, then $r$ is surjective by considering the cubics. If $\xi_2=p\sqcup q$ with $p\neq q$ but $\pi(p)=\pi(q)$, then $r$ is surjective by considering any cubic not passing through $\pi(p)$ and the section corresponding to $R$. If $\xi_2$ is supported at a point $p\notin R$, then $r$ is surjective by considering pull-backs of a cubic not through $\pi(p)$ and a cubic through $\pi(p)$ but not in the direction given by $\xi_2$ on $\Pbb^2$. If $\xi_2$ is supported at a point $p\in R$, then $\xi_2$ is surjective by considering the pull-backs of cubics as above and the section corresponding to $R$.
    
    For $i\geq3$ the restriction map $H^0(\Ocal_S(3))\to H^0(\Ocal_{\xi_i})$ will be surjective only for general $\xi_i\in S^{[i]}$. Thus $\Xcal_4$ is generically a $\Pbb^6$-bundle over $S^{[4]}$, $\Xcal_{1,3}$ is generically a $\Pbb^6$-bundle over $S\times S^{[3]}$, and $\Xcal_8$ is generically a $\Pbb^2$-bundle over $S^{[8]}$. Note that we have the short exact sequence on $S\times \Xcal_i$ (we have an embedding $S\times\Xcal_i\to S\times B\times S^{[i]}$)
    \begin{align*}
        0\to p^*_{12}\Ocal_{S\times B}(-\mathcal{C})\to p^*_{13}\Isc_{\Zcal_i}\to\mathcal{Q}_i \to 0.
    \end{align*}
    Since $\mathcal{Q}_i$ is flat over $\Xcal_i$ and $\vbf(\mathcal{Q}_i|_{S\times\{x\}})=(0,3,-i-9)$ for $x\in \Xcal_i$, we obtain  rational maps
    \begin{eqnarray*}
        f_i:\Xcal_i & \dashrightarrow & M(0,3,-i-9), \\
        (\xi_i\subset C) & \mapsto & \Ical_{\xi/C}.
    \end{eqnarray*}
    We define $F_i:\Xcal_i\dashrightarrow M(0,3,i-9)$ as the composition of $f_i$ with the dual (iso)morphism $$R\mathcal{H}om(-,\Ocal_S(-3))[1]:M(0,3,-i-9)\to M(0,3,i-9).$$
    Now we can prove the four statements of the proposition. We prove (1), (2), and (4) by showing that $F_i$ induces a birational map between $\Xcal_i$ and $Z_i$.
    
\vspace*{2mm}
\noindent(1) We claim that $F_2$ maps $\Xcal_2$ isomorphically onto $Z_2$. First note that $\Ical_{\xi_i/C}$ can be unstable only when $C$ is not integral. One easily sees that $\Ical_{\xi_2/C}$ is always stable by Section~\ref{3H}. To see that $F_2$ is injective, note that $C$ can be recovered as the support of $\Ical_{\xi_2/C}$. We can show that $\RExt^1(\Ical_{\xi_2/C},\Ocal_S(-3))$ has a section by applying the dual functor $R\mathcal{H}om(-,\Ocal_S(-3))[1]$ to \begin{align*}
        0\to \Ocal_S(-3)\to\Ical_{\xi_2}\to \Ical_{\xi_2/C}\to 0.
   \end{align*}
  Corollary \ref{BNinj}(1) (that we will prove in the next section) then shows that $\RExt^1(\Ical_{\xi_2/C},\Ocal_S(-3))$ has a unique section, up to a scalar. This section gives rise to the long exact sequence
\begin{align*}
     0\to \Ocal_S(-3)\to\Ocal_S\to \RExt^1(\Ical_{\xi_2/C},\Ocal_S(-3))\to \Ocal_{\xi_2}\to 0,
\end{align*}
recovering $\xi_2$.
   
To see that $F_2$ induces an injection on the tangent space, we refer the readers to the proof of \cite[Proposition 4.6]{Hel20}. Note that $\Xcal_2$ has an open dense subset $\Xcal_2^\circ$ consisting of pairs $(C,\xi_2)$ where $C$ is smooth. It is clear that $F_2(\Xcal_2^\circ)=\BN_2^0(B^\circ)$. Since $F_2$ is a closed morphism, $F_2(\Xcal_2)=Z_2$. In turn, $Z_2$ is a $\Pbb^8$-bundle over $S^{[2]}$.
    
\vspace*{2mm}
\noindent(2) Let $\Xcal_4^\circ$ be the open subset of $\Xcal_4$ parametrizing pairs $(C,\xi_4)$ where $H^0(\Ical_{\xi_4}(1))=0$. By Section~\ref{3H}, we see that $f_4$ is defined on $\Xcal_4^\circ$. Let $\Xcal_4^{\circ\circ}$ be the open subset of $\Xcal_4^\circ$ where we also require that $C$ is smooth. Then $F_4(\Xcal_4^{\circ\circ})\subset \BN^0_4(B^\circ)$. Since the closure of $\Xcal_4^{\circ\circ}$ in $\Xcal_4^\circ$ is $\Xcal_4^\circ$ itself, we have $F_4(\Xcal_4^\circ)\subset Z_4$. Recall that $Z_2\subset Z_4\subset M(0,3,-5)$. We claim that $F_4(\Xcal_4^\circ)\cap Z_2=\emptyset$. It is easier to prove the claim with $f_4$. The image of $Z_2$ in $M(0,3,-13)$ (which is isomorphic to $M(0,3,-5)$ by $R\mathcal{H}om(-,\Ocal_S(-3))[1]$) is of the form $\Ocal_S(-1)\otimes \RExt^1(\Ical_{\xi_2/C},\Ocal_S(-3))$. Suppose that
\begin{align*}
f_4((C,\xi_4))=\Ical_{\xi_4/C}=\Ocal_S(-1)\otimes \RExt^1(\Ical_{\xi_2/C},\Ocal_S(-3))
\end{align*}
for some $(C,\xi_4)\in \Xcal_4^\circ$ and $\xi_2\subset C$. Applying the dual functor $R\mathcal{H}om(-,\Ocal_S(-3))[1]$ to $0\to \Ical_{\xi_2/C}\to\Ocal_C\to \Ocal_{\xi_2}\to 0$, we obtain 
    \begin{align*}
        0\to \RExt^1(\Ocal_C,\Ocal_S(-3))\to \RExt^1(\Ical_{\xi_2/C},\Ocal_S(-3))\to \RExt^2(\Ocal_{\xi_2},\Ocal_S(-3))\to 0.
    \end{align*}
    Noting that $\RExt^1(\Ocal_C,\Ocal_S(-3))=\Ocal_C$, we obtain an injection $\Ocal_C(-1)\hookrightarrow \Ical_{\xi_4/C}$, contradicting our assumption that $H^0(\Ical_{\xi_4}(1))=0$. Hence $F_4(\Xcal_4^\circ)\subset Z_4\backslash Z_2$. It is easy to see now that $F_4(\Xcal_4^{\circ\circ})=\BN^0_4(B^\circ)\backslash Z_2$. The injectivity of $F_4|_{\Xcal_4^\circ}$ follows from Corollary \ref{BNinj}(2) by arguing as in part (1). Like before, $F_4$ induces an injection on the tangent space. As a result, $Z_4\backslash Z_2$ is generically a $\Pbb^6$-bundle over $S^{[4]}$.
    
\vspace*{2mm}
\noindent(3) Noting the embedding $S\times \Xcal_{1,3}\hookrightarrow S\times B\times S\times S^{[3]}$, we have on $S\times \Xcal_{1,3}$ the short exact sequence
    \begin{align*}
        0\to p^*_{12}\Ocal_{S\times B}(-\mathcal{C})\to p^*_{14}\Ical_{Z_3}\to \mathcal{Q}_3\to 0,
    \end{align*}
     where $\Qcal_3$ is flat and parametrizes $\Ical_{\xi_3/C}$ over $\Xcal_{1,3}$ with $\vbf(\Qcal_3|_{S\times\{x\}})=(0,3,-12)$ for $x\in \Xcal_{1,3}$. Let $\Xcal_{1,3}^\circ$ be the open subset of $\Xcal_{1,3}$ parametrizing triples $(C,p,\xi_3)$ where $H^0(\Ical_{\xi_3}(1))=0$ and $p\notin \xi_3$. For any $(C,p,\xi_3)\in \Xcal^\circ_{1,3}$, there is a unique nontrivial extension of $\Ocal_p$ by $\Ical_{\xi_3/C}$. Let $\Pcal$ denote the universal extension
     \begin{align*}
         0\to \Qcal_3|_{S\times \Xcal_{1,3}^\circ}\to \Pcal\to (p^*_{13}\Ocal_\Delta)|_{S\times \Xcal_{1,3}^\circ}\to 0,
     \end{align*}
     where $\Delta$ denotes the diagonal in $S\times S$. Then $\Pcal$ is flat over $\Xcal^\circ_{1,3}$ and gives a rational map
     \begin{align*}
         f_{1,3}:\Xcal^\circ_{1,3}\dashrightarrow M(0,3,-11).
     \end{align*}
     By Section~\ref{3H}, $f_{1,3}$ is in fact a morphism. Let $F_{1,3}$ be the composition of $f_{1,3}$ with the isomorphism 
     \begin{align*}
         \RHom(-,\Ocal_S(-3))[1]:M(0,3,-11)\to M(0,3,-7).
     \end{align*}
      Let $\Xcal_{1,3}^{\circ\circ}$ be the open subset of $\Xcal_{1,3}^\circ$ where we also require that $C$ is smooth. We have $F_{1,3}(\Xcal_{1,3}^{\circ\circ})\subset Z_{1,3}^\circ$, and hence $F_{1,3}(\Xcal_{1,3}^\circ)\subset Z_{1,3}$.
     
 Let $\Lcal=f_{1,3}((C,p,\xi_3))$ for some $(C,p,\xi_3)\in \Xcal^\circ_{1,3}$. Then we have a distinguished triangle
 \begin{align*}
     \Ical_{\xi_3}\to \Lcal\to \RHom(\Ical_p,\Ocal_S(-3))[1]\to\Ical_{\xi_3}[1].
 \end{align*}
 Applying the dual functor $R\mathcal{H}om(-,\Ocal_S(-2))[1]$ we obtain the distinguished triangle
 \begin{align*}
     \Ical_p(1)\to \RExt^1(\Lcal,\Ocal_S(-2))\to \RHom(\Ical_{\xi_3},\Ocal_S(-2))[1]\to \Ical_p(1)[1].
 \end{align*}
 Since $\Lcal\in M(0,3,-11)$, we have $\RExt^1(\Lcal,\Ocal_S(-2))\in M$. By the proof of Theorem \ref{main} for the seventh wall, $\RExt^1(\Lcal,\Ocal_S(-2))$ is in the exceptional locus in $X_8$ for $g_7$. In particular,
 \begin{align*}
 \RExt^1(\Lcal,\Ocal_S(-2))\in M\cap X_8=M\backslash Z_4.
 \end{align*}
 Thus $F_{1,3}(\Xcal_{1,3}^\circ)\subset Z_{1,3}\backslash Z_4$. In fact, it is easy to see that $F_{1,3}(\Xcal_{1,3}^{\circ\circ})=Z_{1,3}^\circ\backslash Z_4$.
 
 To see that $F_{1,3}|_{\Xcal_{1,3}^\circ}$ is injective, with $\Lcal$ as above we obtain a distinguished triangle
 \begin{align*}
     \Ical_p\to \RExt^1(\Lcal,\Ocal_S(-3))\to \RHom(\Ical_{\xi_3},\Ocal_S(-3))[1]\to\Ical_p[1]
 \end{align*}
 Applying Corollary \ref{BNinj}(3) to $\RExt^1(\Lcal,\Ocal_S(-3))$, we recover $p$ and $\xi_3$. As before, $F_{1,3}$ induces an injection on the tangent space. As a result, $Z_{1,3}\backslash Z_4$ is generically a $\Pbb^6$-bundle over $S\times S^{[3]}$.
 
\vspace*{2mm}
\noindent(4) Let $\Xcal_8^\circ$ be the open subset of $\Xcal_8$ parametrizing pairs $(C,\xi_8)$ where $H^0(\Ical_{\xi_8}(2))=0$ and $H^0(\Ical_{\xi_l}(1))=0$ for any closed subscheme $\xi_l\subset\xi_8$ of length $5\leq l\leq 8$. By Section~\ref{3H}, we see that $f_8$ is defined on $\Xcal_8^\circ$. Let $\Xcal_8^{\circ\circ}$ be the open subset of $\Xcal_8^\circ$ where we also require that $C$ is smooth. Then $F_8(\Xcal_8^{\circ\circ})\subset\BN^0_8(B^\circ)$. Since the closure of $\Xcal_8^{\circ\circ}$ in $\Xcal_8^\circ$ is $\Xcal_8^\circ$ itself, $F_8(\Xcal_8^\circ)\subset Z_8$. 
 
Suppose that $(C,\xi_8)\in \Xcal_8^\circ$. Then $F_8((C,\xi_8))=\RExt^1(\Ical_{\xi_8/C},\Ocal_S(-3))\in M$. By our assumption on $\xi_8$ and the proof of Theorem \ref{main} for the sixth wall, any nontrivial extension $\Fcal$ of the form
 \begin{align*}
     \Ocal_S\to \Fcal\to \RHom(\Ical_{\xi_8},\Ocal_S(-3))[1]
 \end{align*}
 is in the exceptional locus in $X_7$ for $g_6$. Applying the dual functor $\RHom(-,\Ocal_S(-3))$, we obtain
 \begin{align*}
     \Ocal_S(-3)\to \Ical_{\xi_8}\to \RHom(\Fcal,\Ocal_S(-3))[1].
 \end{align*}
 It is easy to see that $\Fcal$ is determined by a nonzero map from $\Ocal_S(-3)$ to $\Ical_{\xi_8}$, thus we have $\RHom(\Fcal,\Ocal_S(-3))[1]\cong\Ical_{\xi_8/C}$ for some $\xi_8\subset C$ and $\Fcal\cong \RExt^1(\Ical_{\xi_8/C},\Ocal_S(-3))$. Since this is true for any such extension, we see that $F_8((C,\xi_8))$ is in the exceptional locus in $X_7$ for $g_6$. In particular,
 \begin{align*}
 \RExt^1(\Ical_{\xi_8/C},\Ocal_S(-3))\in M\cap X_7=M\backslash(Z_4\cup Z_{1,3}).
 \end{align*}
 Hence $F_8(\Xcal_8^\circ)\subset Z_8\backslash(Z_4\cup Z_{1,3})$. In fact, it is easy to see that $F_8(\Xcal_8^{\circ\circ})= \BN^0_4(B^\circ)\backslash(Z_4\cup Z_{1,3})$. By Corollary \ref{BNinj}(4), $F_8$ is injective on an open dense subset of $\Xcal_8^\circ$. As before, $F_8|_{\Xcal_8^\circ}$ induces an injection on the tangent space. As a result, $Z_8\backslash(Z_4\cup Z_{1,3})$ is generically a $\Pbb^2$-bundle over $S^{[8]}$.
\end{proof}
\begin{rem}
The dimension of $Z_2$ is $12$, both $Z_4$ and $Z_{1,3}$ are $14$-dimensional, and the dimension of $Z_8$ is $18$.
\end{rem}

\section{Wall-crossing for $\hilb$ and $M$}\label{wallcross}
By our assumption, $\mathrm{NS}(S)\cong\Z[H]$. For $x\in \R$ and $y\in \R_{>0}$, we use $\sigma_{x,y}$ to denote $\sigma_{xH,yH}$ as defined in Section~\ref{pre-stab}. By \cite[Lemma 6.2]{Bri08}, we obtain a set of (geometric) stability conditions parametrized by an open half plane.
\begin{lem}
For any $x\in\R$ and $y>1$, $\sigma_{x,y}$ is a stability condition on $S$.
\end{lem}
\begin{rem}\label{RestrictionStab}
The restriction $y>1$ is sufficient but not necessary for $\sigma_{x,y}$ to be a stability condition. A more precise requirement is that $Z_{xH,yH}(E)\notin\R_{\leq0}$ for any spherical sheaf $E$ (\cite[Lemma 6.2]{Bri08}). If $\vbf(E)=(r,c_1,s)$, then $Z_{xH,yH}(E)\notin\R_{\leq0}$ amounts to $y>1/r$ when $x=c_1/r$. As a result, we see that for $0<y\leq 1$, $\sigma_{0,y}$ is not a stability condition by considering the spherical object $\Ocal_S[1]$; while for $x\notin\mathbb{Q}$, $\sigma_{x,y}$ is a stability condition for all $y>0$.
\end{rem}
General walls in the plane are (nested) semicircles. For the Hilbert scheme $S^{[n]}=M(1,0,1-n)$ with $n\geq2$, we consider the Mukai vector $\vbf=(1,0,1-n)$. For any $\sigma$ of the form $\sigma_{x,y}$, generic with respect to $\vbf$, $M_\sigma(\vbf)$ is a projective hyperk\"ahler manifold of dimension $2n$. We have the Mukai morphism $\theta_\vbf:\vbf^\perp\xrightarrow{\sim}\mathrm{NS}(M_\sigma(\vbf))\cong \Z^2$. We will omit the subscript from $\theta_\vbf$ when there is no confusion. Bayer and Macr{\`i}~\cite[Section 5]{BM14b} showed that to each wall one can associate a rank two hyperbolic sublattice of $\Halg(S,\Z)$ containing $\vbf$. The following numerical criterion will help us find and characterize the walls for $\vbf$.
\begin{thm}\cite[Theorems 12.1, 12.3]{BM14b}\cite[Remark 2.8]{Cat19}\label{wallcomp}
(1) Divisorial walls are walls whose lattices contain some $\abf\in\Halg(S,\Z)$ such that 
\begin{itemize}
    \item $\abf^2=-2$ and $(\vbf,\abf)=0$, or 
    \item $\abf^2=0$ and $(\vbf,\abf)=1$ or $2$.
\end{itemize}
Together the linear subspaces $\theta(\vbf^\perp\cap \abf^\perp)$ cut out the movable cone $\mathrm{Mov}(S^{[n]})$ in the rational (closed) positive cone $\overline{\mathrm{Pos}}(S^{[n]})_\mathbb{Q}$.\ 

(2) Flopping walls are walls whose lattices contain some $\abf\in\Halg(S,\Z)$ such that 
\begin{itemize}
    \item $\abf^2=-2$ and $1\leq (\vbf,\abf)\leq n-1$, or 
    \item $\abf^2=0$ and $3\leq (\vbf,\abf)\leq n-1$, or
    \item $2\leq \abf^2< \frac{n-1}{2}$ and $2\abf^2+1\leq (\vbf,\abf)\leq n-1$.
\end{itemize}
\end{thm}
To analyze the wall-crossing for $\hilb$, we first compute its movable cone.
\begin{lem}\cite[Proposition 13.1]{BM14b} Let $\tilde{H}=\theta(0,-1,0)$ and $B=\theta(-1,0,-9)$. Then
$$\mathrm{Mov}(\hilb)=\left\langle \tilde{H},\tilde{H}-\frac{1}{3}B\right\rangle.$$
\end{lem}
We note that the two boundaries of the movable cone correspond to the Hilbert-Chow morphism of $\hilb$ and the Lagrangian fibration of $M$, respectively. Recall that $\Mov(\hilb)$ has a finite locally polyhedral chamber decomposition whose chambers corresponds to birational models of $\hilb$. The walls in $\Mov(\hilb)$ are given by rays through $\tilde{H}-\Gamma B$, for certain $\Gamma\in\mathbb{Q}$ satisfying $0< \Gamma<\frac{1}{3}$. To any wall $W$ in $\Stab^\dagger(S)$ for $\vbf=(1,0,-9)$, one can associate a $\Gamma_W$ so that the wall-crossing in $\Stab^\dagger(S)$ and $\Mov(\hilb)$ can be identified (\cite[Example 13.5]{BM14b}). We now claim that for any wall in $\Mov(\hilb)$ given by $\tilde{H}-\Gamma B$, there is a wall in $\Stab^\dagger(S)$ for $\vbf$ associated with $\Gamma$, hence establishing a one-to-one correspondence between walls in $\Stab^\dagger(S)$ for $\vbf$ and in $\Mov(\hilb)$. We believe that the above claim is known to be true among experts (for similar results see \cite[Theorem 10.6, 10.8]{BM14a}), but we could not find a proof for our case in the literature.
\begin{prop}\label{WallBijection}
All minimal models of $\hilb$ arise as moduli spaces of stable objects with Mukai vector $\vbf=(1,0,-9)$, and their birational transformations are induced by crossing walls for $(1,0,-9)$ in $\Stab^\dagger(S)$. \ 

The above statements remain true if we replace $\vbf$ with $\vbf'=(0,3,-1)$.
\end{prop}
\begin{proof}
We will use stability conditions $\sigma_{x,y}$. If a wall for $\vbf$ in $\Stab^\dagger(S)$ is associated to $0<\Gamma<\frac13$, then it intersects the $(x,y)$-plane along the semicircle
\begin{align*}
    \left(x+\frac{1}{\Gamma}\right)^2+y^2=\frac{1}{\Gamma^2}-9.
\end{align*}
Note that if $\Gamma$ is not in the interval $(0,\frac{1}{3})$ (equivalently, if the wall is not a flopping wall), then the wall does not intersect the $(x,y)$-plane.
 Let $\epsilon$ be a small positive irrational number. By Remark \ref{RestrictionStab}, $\sigma_{-3+\epsilon,y}$ is a path of stability conditions for $0<y<+\infty$. Arguing as in the proof of \cite[Theorem 10.8]{BM14a}, we see that there exists a continuous path $l:(0,+\infty)\to \Mov(\hilb)$, starting in the chamber for $M$ and ending on the ray given by $\tilde{H}$, so that $M_{\sigma_{-3+\epsilon,y}}(\vbf)$ is isomorphic to the model corresponding to $l(y)$. Since $\hilb$ has Picard rank two, $l$ goes through all the chambers in $\Mov(\hilb)$. Hence all models of $\hilb$ can be realized as $M_{\sigma_{-3+\epsilon,y}}(\vbf)$ for some $y$, and all walls in $\Mov(\hilb)$ come from walls in the $(x,y)$-plane.\ 
 
 For the last statement of the theorem, note that since $\Phi_*$ preserves the Mukai pairing, the action of $\Phi_*:\mathrm{Stab}^\dagger(S)\to \mathrm{Stab}^\dagger(S)$ induces a bijection between the walls for $\vbf$ and those for $\vbf'=(0,3,-1)$. Then each wall in $\Stab^\dagger(S)$ for $\vbf'$ can be associated to a $\Gamma$. If a wall for $\vbf'$ in $\Stab^\dagger(S)$ is associated to $\Gamma$, with $0<\Gamma<\frac13$, then it intersects the $(x,y)$-plane along the semicircle
\begin{align*}
    \left(x+\frac{1}{6}\right)^2+y^2=\frac{1}{36}+\frac{\Gamma}{6(1-3\Gamma)}.
\end{align*}
The last statement now follows by considering the path of stability conditions $\sigma_{\epsilon',y}$, where $0<y<+\infty$ and $\epsilon'$ is a small positive irrational number, and arguing as in the previous paragraph.
\end{proof}

Applying Theorem \ref{wallcomp}, we give the full list of walls for $\vbf=(1,0,-9)$ in the table below. We remark that the column named `Wall' lists the intersection of the wall corresponding to $\abf$ with the $(x,y)$-plane. This convention will be use throughout the rest of this paper.

\begin{center}
    \begin{tabular}{|c|c|c|c|c|c|}
    \hline
    &&&&&\\[-6pt]
    $\Gamma$ & $\abf$ & $\abf^2$ & $(\vbf,\abf)$ & Wall& Type \\
    &&&&&\\[-6pt]
    \hline
    &&&&\\[-6pt]
    $0$ & $(0,0,-1)$ & 0 & 1 &$x=0$& divisorial \\
    &&&&&\\[-6pt]
    \hline
    &&&&&\\[-6pt]
    $\frac {2}{11}$ & $(1,-1,2)$ & $-2$ & 7 & $x^2+11x+y^2=-9$&flop \\
    &&&&&\\[-6pt]
    \hline
    &&&&&\\[-6pt]
    $\frac 15$ & $(1,-1,1)$ & 0 & 8& $x^2+10x+y^2=-9$ & flop \\
    &&&&&\\[-6pt]
    \hline
    &&&&&\\[-6pt]
    $\frac 29$ & $(1,-1,0)$ & 2 & 9 & $x^2+9x+y^2=-9$& flop \\
    &&&&&\\[-6pt]
    \hline
    &&&&&\\[-6pt]
    $\frac{1}{4}$ & $(0,1,-8)$ & 2 & 8 & $x^2+8x+y^2=-9$& flop \\
    &&&&&\\[-6pt]
    \hline
    &&&&&\\[-6pt]
    $\frac{2}{7}$ & $(0,1,-7)$ & 2 & 7 & $x^2+7x+y^2=-9$& flop \\
    &&&&&\\[-6pt]
    \hline
    &&&&&\\[-6pt]
    $\frac {4}{13}$ & $(1,-2,4)$ & 0 & 5& $x^2+\frac{13}{2}x+y^2=-9$ & flop\\
    &&&&&\\[-6pt]
    \hline
    &&&&&\\[-6pt]
    $\frac {6}{19}$ & $(-1,3,-10)$ & $-2$ & 1& $x^2+\frac{19}{3}x+y^2=-9$ & flop\\
    &&&&&\\[-6pt]
    \hline
    &&&&\\[-6pt]
    $\frac {8}{25}$ & $(-1,4,-16)$ & 0 & 7 & $x^2+\frac{25}{4}x+y^2=-9$& flop\\
    &&&&&\\[-6pt]
    \hline
    &&&&&\\[-6pt]
    $\frac {10}{31}$ & $(2,-5,13)$ & $-2$ & 5& $x^2+\frac{31}{5}x+y^2=-9$ & flop\\
    &&&&&\\[-6pt]
    \hline
    &&&&&\\[-6pt]
    $\frac {14}{43}$ & $(-2,7,-25)$ & $-2$ & 7 & $x^2+\frac{43}{7}x+y^2=-9$& flop\\
    &&&&&\\[-6pt]
    \hline
    &&&&&\\[-6pt]
    $\frac13$ & $(-1,3,-9)$ & 0 & 0 & & Lagrangian fibration\\[-4pt]
    &&&&&\\
    \hline
    \end{tabular}
    \vspace{0.2cm}
    \captionof{table}{Walls of $\mathrm{Mov}(\hilb)$}    
\end{center}

%\iffalse
\begin{center}
\begin{tikzpicture}
    \begin{axis}[axis lines= left, xtick = {-10, -9,...,-1,0}, ytick = {-1,0,...,7,8}, xlabel=$x$-axis, ylabel=$y$-axis, xmin = -10.5, xmax = 0.5, ymin = 0, ymax = 8, samples=500]
   
   \addplot[purple, domain= -4.5:-2] {sqrt(-pow(x,2) - 6.5*x - 9 )};
    \addlegendentry{$\Gamma=4/13$}
    \addplot[blue, domain= -5.31:-1.69] {sqrt(-pow(x,2) - 7*x - 9 )};
    \addlegendentry{$\Gamma=2/7$}
    \addplot[red, domain = -6.66:-1.35]{sqrt(-pow(x,2) - 8*x - 9)};
    \addlegendentry{$\Gamma=1/4$}
    \addplot[orange, domain = -8: -1.146]{sqrt(-pow(x,2) - 9*x - 9)};
    \addlegendentry{$\Gamma=2/9$}
    \addplot[brown, domain = -9:-1]{sqrt(-pow(x,2) - 10*x - 9};
    \addlegendentry{$\Gamma=1/5$}
    \addplot[green, domain = -10.2:-0.8902]{sqrt(-pow(x,2) - 11*x - 9)};
    \addlegendentry{$\Gamma=2/11$}
    \addplot[black,domain=-10.5:2, thick]{1};
 \end{axis}
\end{tikzpicture}
\captionof{figure}{First six walls for $\vbf=(1,0,-9)$}
\end{center}
%\fi

\begin{rem}
\label{rank_one_walls}
In a forthcoming paper~\cite{QS22b} we will introduce the notion of `rank one' walls, which roughly means that in the decomposition $\vbf=\abf+\bbf$ of the Mukai vector at least one of $\abf$ and $\bbf$ must have rank one. In Hellmann's analysis of the rank two Beauville-Mukai system~\cite{Hel20}, all of the walls are of rank one. For the rank three Beauville-Mukai system, notice from the above table that starting from $S^{[10]}$, all of the walls with slope $\Gamma\leq\frac{4}{13}$ have rank one. Similarly, we will see shortly that starting from $M$, all of the walls with slope $\Gamma\geq\frac{4}{13}$ have rank one. The wall $\Gamma=\frac{4}{13}$ has rank one when viewed from both the $S^{[10]}$ and the $M$ sides.

The advantage of having rank one walls is that it is much easier to describe the exceptional loci. However, it is not the case that all walls for degree two K3 surfaces have rank one. Indeed, the rank four Beauville-Mukai system $M(0,4,-1)$ will have walls of higher rank.
\end{rem}

\begin{rem}
In the fifth and sixth rows of the table one might expect the vectors $(1,-1,-1)$ and $(1,-1,-2)$, and indeed they would define these walls. However, these are actually the $\bbf$ vectors, and we instead use the generators $\abf=\vbf-\bbf=(0,1,-8)$ and $(0,1,-7)$, respectively, as they satisfy the constraints of Theorem~\ref{wallcomp}.
\end{rem}

By Proposition \ref{WallBijection}, there are $11$ chambers in $\Mov(\hilb)$, corresponding to $11$ different birational models of $\hilb$. The $0<\Gamma<\frac{2}{11}$ chamber corresponds to $\hilb$ itself while the $\frac{14}{43}<\Gamma<\frac13$ chamber corresponds to $M$. We denote the remaining birational models by $X_i$, $1\leq i\leq9$, with $X_1$ corresponding to the $\frac{2}{11}<\Gamma <\frac15$ chamber.

We also note that for $\Gamma\geq \frac{8}{25}$ the corresponding walls have radii $<1$. We choose to study them from the $M$ side, where the corresponding walls will have radii $>1$. The wall with $\Gamma=\frac{6}{19}$ can be studied from either side. We opt to study it from the $M$ side also. This strategy has the following advantages:
\begin{itemize}
    \item The bound $y>1$ on both sides will make the analysis of the moduli spaces arising from the wall-crossings easier.
    \item Whether we start from $S^{[10]}$ or from $M$, the exceptional loci for walls with radii larger than $1$ have increasing dimensions, making them easier to describe. 
\end{itemize}
We now state the main theorem:

\begin{thm}\label{main}
Let $(S,H)$ be a general polarized K3 surface with $\Pic(S)=\Z [H]$ and $H^2=2$. There are eleven birational models of $S^{[10]}$ or $M\coloneqq M(0,3,-1)$, respectively. They are connected by a chain of flopping contractions
\[ \xymatrix@R-1pc@C-2pc{
& \Bl_{W_0}S^{[10]} \ar[dr]\ar[dl] && \Bl_{\tilde{W}_1}X_1\ar[dr]\ar[dl] && \Bl_{\tilde{W}_2}X_2\ar[dr]\ar[dl] && \Bl_{\tilde{W}_3}X_3\ar[dr]\ar[dl] && \Bl_{\widetilde{W_4\cup \Wcal_0}}X_4\ar[dr]\ar[dl] && \Bl_{\tilde{\Wcal_1}}X_5\ar[dr]\ar[dl]\\
S^{[10]} \ar@{-->}[rr]^{g_0} && X_1 \ar@{-->}[rr]^{g_1} && X_2 \ar@{-->}[rr]^{g_2} && X_3 \ar@{-->}[rr]^{g_3} &&  X_4 \ar@{-->}[rr]^{g_4} &&X_5\ar@{-->}[rr]^{g_5} &&X_6\xrightarrow{\Phi}} \]
\[ \xymatrix@R-1pc@C-2pc{
&  \Bl_{\tilde{Z}_8}X_7\ar[dr]\ar[dl] && \Bl_{\tilde{Z}_{1,3}}X_8\ar[dr]\ar[dl] && \Bl_{\tilde{Z}_4}X_9\ar[dr]\ar[dl] && \Bl_{Z_2}M \ar[dr]\ar[dl]\\
 \xrightarrow{\Phi}X_6' && X_7 \ar@{-->}[ll]^{g_6} && X_8 \ar@{-->}[ll]^{g_7} &&  X_9 \ar@{-->}[ll]^{g_8} &&M \ar@{-->}[ll]^{g_{9}}} \]
for some subvarieties $W_0 \subset W_1\subset\cdots\subset W_4 \subset S^{[10]}$ and $\Wcal_0\subset\Wcal_1\subset \hilb$ such that
\begin{itemize}
\item $W_0$ is a $\Pbb^8$-bundle over $M(0,1,-11)$;
\item $W_{i}\backslash \overline{W_{i-1}}$ is a $\Pbb^{8-i}$-bundle over an open subset of $S^{[i]}\times M(0,1,i-11)$ for $i=1,2,3$;
\item $W_4\backslash(\Wcal_0\cup \overline{W_3})$ is a $\Pbb^4$ bundle over an open subset in $S^{[4]}\times M(0,1,-7)$, while  $\Wcal_0\backslash \overline{W_4}$ is generically a $\Pbb^5$-bundle over $M^{st}(0,2,-14)$;
\item $\Wcal_1\backslash(\Wcal_0\cup\overline{W_4})$ is generically a $\Pbb^4$-bundle over $S\times M(0,2,-13)$;
\end{itemize}
and closed subvarieties $Z_2 \subset Z_4 \subset Z_8\subset M$ and $Z_2\subset Z_{1,3}\subset Z_8\subset M$ such that
\begin{itemize}
\item $Z_2$ is a $\Pbb^8$-bundle over $S^{[2]}$;
\item $Z_4$ is generically isomorphic to a $\Pbb^6$-bundle over $S^{[4]}$;
\item $Z_{1,3}$ is generically isomorphic to a $\Pbb^6$-bundle over $S\times S^{[3]}$;
\item $Z_8$ is generically isomorphic to a $\Pbb^2$-bundle over $S^{[8]}$.
\end{itemize}
Here $\widetilde{\bullet}$ denotes the strict transform of the set $\bullet$ under suitable birational maps. The model $X_6$ is isomorphic to $X_6'$ via $\Phi$.
\end{thm}

\begin{rem}
    One technical difficulty is the potential existence of totally semistable walls. Fortunately, in our case we can avoid totally semistable walls by choosing the path of wall-crossing properly (as we will do in the next proof). This issue was addressed in \cite[Section 5]{QS22b} in greater generality and we refer the readers there for proofs.
\end{rem}

\begin{proof}
There are ten flopping walls in Table $1$. We label them the $i$-th wall for $i=0,1,\ldots, 9$ with increasing $\Gamma$ (for example, the wall with $\Gamma=\frac{2}{11}$ is the $0$-th wall). We refer the reader to \cite[Section 14]{BM14b} for some notation and results which will be useful for our analysis of these flopping walls. We study the first six flopping walls from the $\hilb$ side. Consider the path of stability conditions $\sigma_t\coloneqq\sigma_{-3,t}$ for $t\in(1,+\infty)$. For $0\leq i\leq 5$ the path intersects the $i$-th wall at $t=t_i$, where $t_i=\sqrt{15-3i}$ for $0\leq i\leq 4$ and $t_5=\sqrt{3/2}$. Next we describe the wall-crossings along $\sigma_t$.

The $0$-th wall corresponds to the decomposition
\begin{align*}
    (1,0,-9)=(1,-1,2)+(0,1,-11).
\end{align*}
By the appendix, near $t=t_0$, $M_{\sigma_t}(0,1,-11)=M_H(0,1,-11)$.
An ideal sheaf $\Ical_\xi$ is in the exceptional locus $E_0$ of $g_0$ in $\hilb$ if and only if it fits into the short exact sequence
\begin{align*}
    0\to \Ocal_S(-1)\to \Ical_\xi\to \Ecal_0\to 0,
\end{align*}
where $\Ecal_0\in M_H(0,1,-11)$. This is equivalent to $\xi\in W_0$. Thus $g_0$ is the flop of $W_0$ in $\hilb$. \\

For $i=1,2,$ and $3$, the $i$-th wall corresponds to the decomposition of Mukai vectors
\begin{align*}
    (1,0,-9)=(1,-1,2-i)+(0,1,-11+i).
\end{align*}
By the appendix, near $t=t_i$, we have $M_{\sigma_t}(1,-1,2-i)\cong S^{[i]}$ and $M_{\sigma_t}(0,1,-11+i)=M_H(0,1,-11+i)$. Hence an ideal sheaf $\Ical_\xi$ is in the exceptional locus $E_i$ of $g_i$ in $X_i$ if and only if $\Ical_\xi\notin \overline{W_{i-1}}$ and it fits into the short exact sequence
\begin{align*}
    0\to \Ical_{\zeta_i}(-1)\to \Ical_\xi\to \Ecal_i\to 0,
\end{align*}
where $\zeta_i\in S^{[i]}$ and $\Ecal_i\in M_H(0,1,-11+i)$. This is equivalent to $\xi\in \overline{W_i}\backslash\overline{W_{i-1}}$. Moreover, the exceptional locus of $g_i$ in $X_i$ is the strict transform of $\overline{W_{i}}$, which we denote by $\tilde{W_i}$. \\

The study of the fourth wall is more complicated. The fourth wall corresponds to the decompositions of Mukai vectors
\begin{align*}
    (1,0,-9)&=(1,-1,-2)+(0,1,-7)\\
    &=(1,-2,5)+(0,2,-14)\\
    &=(1,-2,5)+(0,1,-7)+(0,1,-7),
\end{align*}
where the last decomposition is a refinement of the previous two. By \cite[Section 14]{BM14b}, the exceptional locus $E_4$ of $g_4$ in $X_4$ has a stratification $E_4=E_4^1\coprod E_4^2\coprod E_4^3$ into locally closed subsets, where $E_4^j$ corresponds to the $j$-th line in the decomposition above and $E_4^3$ is in the closure of both $E_4^1$ and $E_4^2$. Thus $E_4$ has two irreducible components given by $E_4^1\cup E_4^3$ and $E_4^2\cup E_4^3$. By the appendix, near $t=t_4$, we have $M_{\sigma_t}(1,-2,5)=\{\Ocal_S(-2)\}$, $M_{\sigma_t}(0,1,-7)=M_H(0,1,-7)$, and $M^{st}_{\sigma_t}(0,2,-14)=M^{st}_H(0,2,-14)$. For $t>t_4$, $M_{\sigma_t}(1,-1,-2)=M_H(1,-1,-2)\cong S^{[4]}$, but $M^{st}_{\sigma_{t_4}}(1,-1,-2)\cong S^{[4]}\backslash\{\xi_4\;|\;h^0(\Ical_{\xi_4}(1))\geq 1\}$. Hence an ideal sheaf $\Ical_\xi$ is in $E_4^3$ if and only if $\xi\notin \overline{W_3}$ and for $t$ near $t_4$ with $t>t_4$ it has a HN-filtration (for $\sigma_t$)
    \[
  \xymatrix{
  0 \ar[r] & \Ocal_S(-2) \ar[r] \ar[d] & F \ar[r] \ar[d] & \Ical_\xi \ar[d]\\
  & \Ocal_S(-2) \ar@{-->}[lu] & \Ecal_6 \ar@{-->}[lu] & \Ecal'_6,
  \ar@{-->}[lu] }
  \]
where $\Ecal_6$, $\Ecal_6'\in M_H(0,1,-7)$. By stability, $F$ is a nontrivial extension of $\Ecal_6$ by $\Ocal_S(-2)$, and hence $F=\Ical_{\zeta_4}(-1)$ for some $\zeta_4\in S^{[4]}$. This is equivalent to $\xi\in (\overline{W_4}\cap\Wcal_0)\backslash \overline{W_3}$.

Now $\Ical_\xi\in E_4^1$ if and only if $\xi\notin \overline{W_3}$ and it fits into a short exact sequence
\begin{align*}
    0\to \Ical_{\eta_4}(-1)\to \Ical_\xi\to \Ecal_6''\to 0,
\end{align*}
where $\eta_4\in M^{st}_{\sigma_{t_4}}(1,-1,-2)=S^{[4]}\backslash\{\xi_4\;|\;h^0(\Ical_{\xi_4}(1))\geq 1\}$ and $\Ecal_6''\in M_H(0,1,-7)$. This is equivalent to $\xi\in \overline{W_4}\backslash(\Wcal_0\cup\overline{W_3})$.

Similarly, $\Ical_\xi\in E_4^2$ if and only if $\xi\notin \overline{W_3}$ and it fits into a short exact sequence
\begin{align*}
    0\to \Ocal_S(-2)\to \Ical_\xi\to \Fcal\to 0,
\end{align*}
where $\Fcal\in M^{st}_{\sigma_t}(0,2,-14)=M^{st}_H(0,2,-14)$. This is equivalent to $\xi\in \Wcal_0\backslash\overline{W_4}$. Altogether, we see that $E_4$ is the strict transform of $W_4\cup\Wcal_0$.\\

The fifth wall corresponds to the decompositions of Mukai vectors
\begin{align*}
    (1,0,-9)&=(1,-2,4)+(0,2,-13)\\
    &=(1,-2,4)+(1,-2,4)+(-1,4,-17).
\end{align*}
By \cite[Section 14]{BM14b}, the exceptional locus $E_5$ in $X_5$ for $g_5$ has a stratification $E_5=E_5^1\coprod E_5^2$ where $E_5^j$ corresponds to the $j$-th line of the decomposition above. Note that $E_5^2$ is closed in $E_5$. Near $t=t_5$, $M_{\sigma_t}(1,-2,4)\cong S$ by $\Ical_p(-2)\mapsto p$, and for $t>t_5$ we have $M_{\sigma_t}(0,2,-13)=M_H(0,2,-13)$ and
\begin{align*}
M^{st}_{\sigma_{t_5}}(0,2,-13)=M_H(0,2,-13)\backslash\{\Ical_{q/Q}(-2)\;|\;Q\in|2H|\text{ and } q\in Q\}.
\end{align*}
On the other hand, near $t=t_5$, $\{\Ical_{q/Q}(-2)\;|\; Q\in|2H|\text{ and } q\in Q\}$ parametrizes precisely the extensions of $\Ocal_S(-4)[1]$ by $\Ical_q(-2)$. Altogether, an ideal sheaf $\Ical_\xi$ is in $E_5$ if and only if $\xi\notin \overline{W_4}\cup\Wcal_0$ and it fits into a short exact sequence
\begin{align*}
    0\to \Ical_p(-2)\to \Ical_\xi\to \Fcal'\to 0,
\end{align*}
where $p\in S$ and $\Fcal'\in M_H(0,2,-13)$. This is equivalent to $\xi \in\overline{\Wcal_1}\backslash(\Wcal_0\cup\overline{W_4}).$\\ 

We analyze the remaining walls from the $M$ side. Let $\vbf'=(0,3,-1)=\Phi_*\vbf$. The next table describes the remaining walls from the $M$ side, with $\abf'=\Phi_*\abf$. Note that all circles have radii larger than $1$.
\begin{center}
    \begin{tabular}{|c|c|c|c|c|c|}
    \hline
    &&&&&\\[-6pt]
    $\Gamma$ & $\abf'$  & $(\abf')^2$ & $(\vbf',\abf')$ & Wall& Type \\
    &&&&&\\[-6pt]
    
    \hline

    &&&&&\\[-6pt]
    $\frac {6}{19}$ & $(1,0,1)$ & $-2$ & $1$ & $\left(x+\frac{1}{6}\right)^2+y^2=\left(\frac{\sqrt{37}}{6}\right)^2$ & flop\\
    &&&&&\\[-6pt]
    \hline
    &&&&&\\[-6pt]
    $\frac {8}{25}$ & $(1,1,1)$ & $0$ & $7$ & $\left(x+\frac{1}{6}\right)^2+y^2=\left(\frac{7}{6}\right)^2$& flop\\
    &&&&&\\[-6pt]
    \hline
    &&&&&\\[-6pt]
    $\frac {10}{31}$ & $-(1,-1,2)$ & $-2$ & $5$ & $\left(x+\frac{1}{6}\right)^2+y^2=\left(\frac{\sqrt{61}}{6}\right)^2$ & flop\\
    &&&&&\\[-6pt]
    \hline
    &&&&&\\[-6pt]
    $\frac {14}{43}$ & $(1,1,2)$ & $-2$ & $7$ & $\left(x+\frac{1}{6}\right)^2+y^2=\left(\frac{\sqrt{85}}{6}\right)^2$& flop\\
    &&&&&\\[-6pt]
    \hline
    &&&&&\\[-6pt]
    $\frac13$ & $(0,0,1)$ & $0$ & $0$ & & Lagrangian fibration\\[-4pt]
    &&&&&\\
    \hline
    \end{tabular}
    \vspace{0.2cm}
    \captionof{table}{Walls for $M$ with radii larger than $1$} 
\end{center}

%\iffalse
\begin{center}
\begin{tikzpicture}
    \begin{axis}[axis lines= left, xtick = {-2, -1.5,...,1,1.5}, ytick = {0,0.5,...,2,2.5}, xlabel=$x$-axis, ylabel=$y$-axis, xmin = -2, xmax = 1.5, ymin = 0, ymax = 2.5, samples=500]
   
    \addplot[blue, domain= -1.70:1.37] {sqrt(-pow(x,2) - x/3 +7/3 )};
    \addlegendentry{$\Gamma=14/43$}
   \addplot[red, domain = -1.468:1.135]{sqrt(-pow(x,2) - x/3 +5/3)};
    \addlegendentry{$\Gamma=10/31$}
    \addplot[orange, domain = -1.34: 1]{sqrt(-pow(x,2) - x/3 +4/3)};
    \addlegendentry{$\Gamma=8/25$}
    \addplot[brown, domain = -1.19:0.847]{sqrt(-pow(x,2) - x/3 +1};
    \addlegendentry{$\Gamma=6/19$}
    \addplot[black,domain=-2:1.5]{1};
    \end{axis}
\end{tikzpicture}
\captionof{figure}{Last four walls for $\vbf'=(0,3,-1)$}
\end{center}
%\fi

We keep the labeling of the walls from the $\hilb$ side; thus the walls corresponding to $\Gamma =\frac{14}{43}$ and $\frac{10}{31}$ will be called the ninth and eighth walls, respectively. Consider the path $\sigma'_t:=\sigma_{-\frac16,t}$ for $t\in (1,+\infty)$. For $6\leq i\leq 9$, this path crosses the $i$-th wall at $t'_i$, where $t'_i$ is the radius of the $i$-th wall. We now describe the wall-crossings along $\sigma'_t$.

The ninth wall corresponds to the decomposition of Mukai vectors
\begin{align*}
    (0,3,-1)=(1,1,2)+(-1,2,-3).
\end{align*}
By the appendix, near $t=t'_9$, $S^{[2]}\cong M_{\sigma'_t}(-1,2,-3)$ by $\xi_2\mapsto \RHom(\Ical_{\xi_2},\Ocal_S)(-2)[1]$. Hence $\Ecal$ is in the exceptional locus $E_9$ of $g_9$ in $M$ if and only if it fits into the (non-split) distinguished triangle
\begin{align*}
    \Ocal_S(1)\to\Ecal\to \RHom(\Ical_{\xi_2},\Ocal_S)(-2)[1].
\end{align*}
Note that there exists a short exact sequence
\begin{align*}
    0\to \Ext^1(\RHom(\Ical_{\xi_2},\Ocal_S(-2))[1],\Ocal_S(1))\to \Hom(\Ocal_S(-2),\Ocal_S(1))\to \Ext^2(\Ocal_{\xi_2},\Ocal_S(1))\to 0.
\end{align*}
So $\Ecal$ is in the exceptional locus if and only if it fits into the exact sequence
\begin{align*}
    0\to \Ocal_S(-2)\xrightarrow{s_9} \Ocal_S(1)\to \Ecal\to \Ocal_{\xi_2}\to 0
\end{align*}
where $s_9\not\equiv0$ and $\xi_2$ is contained in the cubic $\{s_9=0\}$.
When $s_9$ corresponds to a smooth cubic $C\in|3H|$, this means precisely that $\Ecal(-1)\in \BN^0_2(B^\circ)$. Since both the exceptional locus and $Z_2$ are closed and irreducible, they must be the same.\\

The eighth wall corresponds to the decomposition of Mukai vectors
\begin{align*}
    (0,3,-1)=(1,2,1)+(-1,1,-2).
\end{align*}
By the appendix, we see that there are two birational models for $S^{[4]}$. Let ${}^{\sharp}S^{[4]}$ denote the other model, which is obtained by flopping at the locus parametrizing length four subschemes on a line. Then near $t=t_8'$, $M_{\sigma'_t}(1,2,1)\cong{}^{\sharp}S^{[4]}$. We note that all points in ${}^{\sharp}S^{[4]}$ corresponds to sheaves and they are torsion free if and only if they are not in the exceptional locus of the flop of $S^{[4]}$.

By \cite[Section 14]{BM14b}, the exceptional locus $E_8$ of $g_8$ in $X_9$ is irreducible and $\Ecal\in E_8$ if and only if it fits into a (non-split) distinguished triangle
\begin{align*}
   I\to \Ecal\to \Ocal_S(-1)[1],
\end{align*}
where $I\in M_{\sigma'_8}(1,2,1)$. In particular, there exists an open dense subset $E_8^\circ\subset E_8$ consisting of $\Ecal$ fitting into a short exact sequence
\begin{align*}
    0\to \Ocal_S(-1)\xrightarrow{s_8} \Ical_{\xi_4}(2)\to \Ecal\to 0,
\end{align*}
where $h^0(\Ical_{\xi_4}(1))=0$ and $s_8$ corresponds to a smooth curve in $|3H|$. It is easy to see that $\Ecal\in E_8^\circ$ if and only if $\Ecal\in \BN^0_4(B^\circ)\backslash Z_2$. Thus the exceptional locus of $g_8$ in $X_9$ is the strict transform of $Z_4$.\\

The seventh wall corresponds to the decomposition of Mukai vectors
\begin{align*}
    (0,3,-1)=(1,1,1)+(-1,2,-2).
\end{align*}
By the appendix, near $t=t'_7$, $S\cong M_{\sigma'_t}(1,1,1)$ by $p\mapsto \Ical_p(1)$ and $S^{[3]}\cong M_{\sigma'_t}(-1,2,-2)$ generically by $\xi_3\mapsto \RHom(\Ical_{\xi_3},\Ocal_S)(-2)[1]$ for $\xi_3$ not on a line. The exceptional locus $E_7$ of $g_7$ in $X_8$ is irreducible and $\Ecal\in E_8$ if and only if it fits into a (non-split) distinguished triangle
\begin{align*}
    \Ical_p(1)\to\Ecal\to J,
\end{align*}
where $p\in S$ and $J\in M_{\sigma'_7}(-1,2,-2)$. In particular, there exists an open dense subset $E_7^\circ\subset E_7$ consisting of $\Ecal$ fitting into
\begin{align*}
    \RHom(\Ical_{\xi_3},\Ocal_S)(-2)\xrightarrow{s_7}\Ical_p(1)\to\Ecal,
\end{align*}
where $h^0(\Ical_{\xi_3}(1))=0$, $p\notin\xi_3$, and $s_7$ corresponds to a smooth cubic in $|3H|$ (assuming the first two conditions, $s_7$ corresponds to a cubic containing $\xi_3\cup\{p\}$). It is easy to see that $\Ecal\in E_7^\circ$ if and only if $\Ecal\in Z^\circ_{1,3}\backslash Z_4$. Thus the exceptional locus of $g_7$ in $X_8$ is the strict transform of $Z_{1,3}$.\\

The sixth wall corresponds to the decompositions of Mukai vectors
\begin{align*}
    (0,3,-1)&=(1,0,1)+(-1,3,-2)\\
    &=(1,0,1)+(1,0,1)+(-2,3,-3).
\end{align*}
We note that for $t'$ near $t'_6$, $M_{\sigma_{t'}}(-1,3,-2)$ is birational to $S^{[8]}$. 
Arguing as we did for the fifth wall, we see that $\Ecal$ is in the exceptional locus $E_6$ of $g_6$ in $X_7$ if and only if it fits into a (non-split) distinguished triangle
\begin{align*}
    \Ocal_S\to \Ecal\to K,
\end{align*}
where $K\in M_{\sigma_{t'}}(-1,3,-2)$ for $t'>t'_6$ and $t'$ close to $t'_6$. By the appendix, there exists an open dense subset $E_6^\circ\subset E_6$  consisting of $\Ecal$ fitting into 
\begin{align*}
     \RHom(\Ical_{\xi_8},\Ocal_S(-3))\xrightarrow{s_6}\Ocal_S\to \Ecal,
\end{align*}
where $h^0(\Ical_{\xi_8}(2))=0$, $h^0(\Ical_{\xi_l}(1))=0$ for any subscheme $\xi_l\subset\xi_8$ of length $5\leq l\leq 8$ (these two conditions hold if and only if $\xi_8$ is not in the exceptional loci of any of the first four flopping walls for $S^{[8]}$), and $s_6$ corresponds to a smooth cubic in $|3H|$. It is easy to see that $\Ecal\in E_6^\circ$ if and only if $\Ecal\in \BN^0_8\backslash( Z_4\cup Z_{1,3})$. Thus the exceptional locus of $g_6$ in $X_7$ is the strict transform of $Z_8$.\\

Lastly, we show that $\Phi$ induces an isomorphism $X_6\to X_6'$. Note that the point $(x,y)=(-3,6/5)$ is below the wall with $\Gamma=\frac{4}{13}$ on the $\hilb$ side and satisfies $y>1$ (in particular, it is above the wall with $\Gamma=\frac{6}{19}$); thus $X_6=M_{\sigma_{(-3,6/5)}}$. By Lemma \ref{grpact}, $\Phi$ induces an isomorphism 
\begin{align*}
    M_{\sigma_{(-3,6/5)}}(1,0,-9)\cong M_{\Phi_*(\sigma_{(-3,6/5)})}(0,3,-1).
\end{align*}
We note that $\Phi_*(-1,3,-9)=(0,0,1)$ and $\sigma_{(-3,6/5)}$-stable objects with Mukai vector $(-1,3,-9)$ are of the form $\RHom(\Ical_p,\Ocal_S(-3))[1]$ for $p\in S$. It is easy to check that
$$\Phi(\RHom(\Ical_p,\Ocal_S(-3))[1])\cong \Ocal_p,$$
and thus all skyscraper sheaves are $\Phi_*(\sigma_{(-3,6/5)})$-stable and $\Phi_*(\sigma_{(-3,6/5)})$ is geometric. Since the action of $\widetilde{\mathrm{GL}}_2^+(\R)$ does not change objects in the moduli space, $M_{\Phi_*(\sigma_{(-3,6/5)})}(0,3,-1)=M_{\sigma_{x',y'}}(0,3,-1)$ for some $x'\in\R$ and $y'>0$. It is clear that $\sigma_{x',y'}$ must lie in a chamber between the walls corresponding to $\Gamma=\frac{4}{13}$ and $\Gamma=\frac{6}{19}$ on the $M$ side. Moreover, 
one can check that $M_{\sigma_{x',y'}}(0,3,-1)=X_6'$. To conclude, we have an isomorphism
\begin{align*}
   X_6= M_{\sigma_{(-3,6/5)}}(1,0,-9)\xrightarrow{\Phi} M_{\Phi_*(\sigma_{(-3,6/5)})}(0,3,-1)=M_{\sigma_{x',y'}}(0,3,-1)=X_6'
\end{align*}
induced by $\Phi$.
\end{proof}

\begin{cor}\label{BNinj}
For $i=2,4$, and $8$, let $\Ecal_i\in M_H(0,3,i-9)$.
\begin{enumerate}
    \item If $h^0(\Ecal_2)\neq0$, then $h^0(\Ecal_2)=1$.
    \item If $h^0(\Ecal_4)\neq0$ and $h^0(\RExt^1(\Ecal_4,\Ocal_S(-2)))=0$, then $h^0(\Ecal_4)=1$.
    \item If $h^0(\Ecal_2)=0$, $h^0(\RExt^1(\Ecal_2,\Ocal_S(-2)))=0$, and $\Ecal_2\in Z_{1,3}$, then there exist a unique $p\in \Supp(\Ecal_2)$ such that $\Hom(\Ical_p,\Ecal_2)\neq0$. In fact, $\mathrm{hom}(\Ical_p,\Ecal_2)=1$.
    \item If $h^0(\Ecal_8)\neq0$, $h^0(\RExt^1(\Ecal_8,\Ocal_S(-1)))=0$, and $\Ecal_8(-1)\notin Z_{1,3}$, then $h^0(\Ecal_8)\leq 2$.
\end{enumerate}
\end{cor}
\begin{proof}
By looking at the possible decompositions of the Mukai vectors into effective classes at the walls (see \cite[Proposition 5.5]{BM14b}), one can obtain Brill-Noether type bounds (for example, see \cite[Section 3.2]{BM22}).

(1) Suppose $h^0(\Ecal_2)\neq 0$. Then $\Ecal_2(1)\in M$ is in the exceptional locus for the ninth wall. By the decomposition corresponding to the wall, $h^0(\Ecal_2)=1$.

(2) Suppose $h^0(\Ecal_4)\neq 0$ and $h^0(\RExt^1(\Ecal_4,\Ocal_S(-2)))=0$. We have $$\Hom(\RExt^1(\Ecal_4,\Ocal_S(-1)),\Ocal_S(-1)[1])\neq0$$
by the local and global Ext spectral sequence \cite[Equation (3.16)]{Huy06}. Thus $\RExt^1(\Ecal_4,\Ocal_S(-1))\in M$ is in the exceptional locus of $g_8$ in $X_9$. By the decomposition corresponding to the eighth wall, $h^0(\Ecal_4)=1$.

(3) If $h^0(\Ecal_2)=0$, $h^0(\RExt^1(\Ecal_2,\Ocal_S(-2)))=0$, and $\Ecal_2\in Z_{1,3}$, then $\Ecal_2(1)\in M$ is in the exceptional locus of $g_7$ in $X_8$. Our claim now follows from the decomposition corresponding to the seventh wall.

(4) If $h^0(\Ecal_8)\neq 0$, $h^0(\RExt^1(\Ecal_8,\Ocal_S(-1)))=0$, and $\Ecal_8(-1)\notin Z_{1,3}$, then $\Ecal_8\in M$ is in the exceptional locus of $g_6$ in $X_7$. By the decompositions of the Mukai vector corresponding to the sixth wall, $h^0(\Ecal_8)\leq 2$.
\end{proof}

\begin{cor}
Let $\Xcal_4$, $\Xcal_4^\circ$, $\Xcal_{1,3}$, $\Xcal_{1,3}^\circ$, $\Xcal_8$, and $\Xcal_8^\circ$ be as defined in Proposition \ref{BNforM}. Then 
\begin{enumerate}
    \item $\Xcal_4^\circ$ is a $\Pbb^6$-bundle over the open subset $\Ucal_4:=\{\xi_4\;|\; h^0(\Ical_{\xi_4}(1))=0\}$ in $S^{[4]}$,
    \item $\Xcal_{1,3}^\circ$ is a $\Pbb^6$-bundle over the open subset $\Ucal_{1,3}:=\{(p,\xi_3)\;|\; h^0(\Ical_{\xi_3}(1))=0 \mbox{ and } p\notin \xi_3\}$ in $S\times S^{[3]}$,
    \item $\Xcal_8^\circ$ contains an open subset which is isomorphic to a $\Pbb^2$-bundle over the open subset
      \begin{align*}
        \Ucal_8:=\left\{\xi_8\;\left|\begin{array}{c}
	 h^0(\Ical_{\xi_8}(2))=0,\mbox{ }h^1(\Ical_{\xi_8}(3))=0,\mbox{ and } \\
          h^0(\Ical_{\xi_l}(1))=0\mbox{ for any }\xi_l\subset\xi_8\mbox{ of length }5\leq l\leq 8
          \end{array}\right.\right\}
    \end{align*}
    in $S^{[8]}$.
\end{enumerate}
As a result, $Z_i$ contains an open subscheme which is a $\Pbb^{10-i}$-bundle over an open subset of $S^{[i]}$ for $i=4$ and $8$, and $Z_{1,3}$ contains an open subscheme which is a $\Pbb^6$-bundle over an open subset of $S\times S^{[3]}$. 
\end{cor}
\begin{proof}
Recall that we have embeddings $\Xcal_i\hookrightarrow B\times S^{[i]}$ and $\Xcal_{1,3}\hookrightarrow B\times S\times S^{[3]}$.

(1) By definition $\Xcal_4^\circ=p_2^{-1}(\Ucal_4)$. It suffices to show that $h^0(\Ical_{\xi_4}(3))=7$ for any $\xi_4\in\Ucal_4$. By the analysis of the eighth wall, we see that $\Ocal_S(-1)[1]$ and $\Ical_{\xi_4}(2)$ are $\sigma'_8$-stable with the same phase. Thus
\begin{align*}
    h^0(\Ical_{\xi_4}(3))=\dim\Ext^1(\Ocal_S(-1)[1],\Ical_{\xi_4}(2))=\big(\vbf(\Ocal_S(-1)[1]),\vbf(\Ical_{\xi_4}(2))\big)=7.
\end{align*}

(2) Note that $\Xcal^\circ_{1,3}=p_{23}^{-1}(\Ucal_{1,3})$. It suffices to show that $h^0(\Ical_{\xi_3\cup \{p\}}(3))=7$ for any $(p,\xi_3)\in\Ucal_{1,3}$. By the analysis of the seventh wall, we see that $\Ical_p(1)$ and $\RHom(\Ical_{\xi_3},\Ocal_S(-2))[1]$ are $\sigma'_7$-stable with the same phase. Thus
\begin{align*}
    \dim\Ext^1(\RHom(\Ical_{\xi_3},\Ocal_S(-2))[1],\Ical_{p}(1))=\big(\vbf(\RHom(\Ical_{\xi_3},\Ocal_S(-2))[1]),\vbf(\Ical_{p}(1))\big)=7.
\end{align*}
On the other hand, we have the short exact sequence
\begin{align*}
    0\to \Ext^1(\RHom(\Ical_{\xi_3},\Ocal_S(-2))[1],\Ical_{p}(1))\to \Hom(\Ocal_S(-2),\Ical_p(1))\to \Ext^2(\Ocal_{\xi_3},\Ical_p(1))\to 0.
\end{align*}
Noting $\Ext^2(\Ocal_{\xi_3},\Ical_p(1))=\Hom(\Ical_p(1),\Ocal_{\xi_3})^*$, we see that $\Ext^1(\RHom(\Ical_{\xi_3},\Ocal_S(-2))[1],\Ical_{p}(1))$ is the subspace of cubics containing both $p$ and $\xi_3$, hence it has dimension $h^0(\Ical_{\xi_3\cup\{p\}}(3))$.

(3) Note that $p_2^{-1}(\Ucal_8)$ is an open subset of $\Xcal_8^\circ$. It suffices to show that $h^0(\Ical_{\xi_8}(3))=3$ for any $\xi_8\in\Ucal_8$. By the analysis of the sixth wall, we see that $\Ocal_S$ and $\RHom(\Ical_{\xi_8},\Ocal_S(-3))[1]$ are $\sigma'_6$-stable with the same phase (we note that the new condition $h^1(\Ical_{\xi_8}(3))=0$ is to guarantee that $\RHom(\Ical_{\xi_8},\Ocal_S(-3))[1]$ is $\sigma_6'$-stable, instead of just $\sigma_6'$-semistable). Thus
\begin{align*}
    \dim\Ext^1(\RHom(\Ical_{\xi_8},\Ocal_S(-3))[1],\Ocal_S)=\big(\vbf(\RHom(\Ical_{\xi_8},\Ocal_S(-3))[1]),\vbf(\Ocal_S)\big)=3.
\end{align*}
On the other hand, we have the short exact sequence
\begin{align*}
    0\to \Ext^1(\RHom(\Ical_{\xi_8},\Ocal_S(-3))[1],\Ocal_S)\to \Hom(\Ocal_S(-3),\Ocal_S)\to \Ext^2(\Ocal_{\xi_8},\Ocal_S)\to 0.
\end{align*}
Noting that $\Ext^2(\Ocal_{\xi_8},\Ocal_S)=\Hom(\Ocal_S,\Ocal_{\xi_8})^*$, we see that $\Ext^1(\RHom(\Ical_{\xi_8},\Ocal_S(-3))[1],\Ocal_S)$ is the subspace of cubics containing $\xi_8$, hence it has dimension $h^0(\Ical_{\xi_8}(3))$.
\end{proof}

\section{Appendix}
Here we collect some results about wall-crossings for the moduli spaces which have appeared in previous sections. 
\subsection{Walls for $S^{[2]}$}
We have $\vbf=(1,0,-1)$.
\begin{lem}\cite[Proposition 13.1]{BM14b}
Let $\tilde{H}=\theta(0,-1,0)$ and $B=\theta(-1,0,-1)$. Then
$$\mathrm{Mov}(S^{[2]})=\langle \tilde{H},\tilde{H}-B\rangle.$$
\end{lem}
The full list of walls is given in the table below.

\hspace*{3mm}
\begin{center}
    \begin{tabular}{|c|c|c|c|c|c|}
    \hline
    &&&&&\\[-6pt]
    $\Gamma$ & $\abf$ & $\abf^2$ & $(\vbf,\abf)$ & Wall& Type \\
    &&&&&\\[-6pt]
    \hline
    &&&&&\\[-6pt]
    $0$ & $(0,0,-1)$ & 0 & 1 &$x=0$ & divisorial \\
    &&&&&\\[-6pt]
    \hline
    &&&&&\\[-6pt]
    $\frac23$ & $(-1,1,-2)$ & $-2$ & 1 & $\left(x+\frac{3}{2}\right)^2+y^2=\left(\frac{\sqrt{5}}{2}\right)^2$& flop  \\
    &&&&&\\[-6pt]
    \hline
    &&&&&\\[-6pt]
    $1$ & $(-1,1,-1)$ & 0 & 0 && Lagrangian fibration \\[-6pt]
    &&&&&\\
    \hline
    \end{tabular}
    \vspace{0.2cm}
    \captionof{table}{Walls of $\mathrm{Mov}(S^{[2]})$} 
\end{center}
The only other nontrivial birational model for $S^{[2]}$ is $M(0,1,-1)$, obtained by performing a flop of $S^{[2]}$ along the locus parametrizing $\xi_2\in S^{[2]}$ through which there passes a pencil of lines. From the $M(0,1,-1)$ side, the flopping wall is given by $\left(x+\frac{1}{2}\right)^2+y^2=\left(\frac{\sqrt{5}}{2}\right)^2$.

\subsection{Walls for $S^{[3]}$} We have $\vbf=(1,0,-2)$
\begin{lem}\cite[Proposition 13.1]{BM14b}
Let $\tilde{H}=\theta(0,-1,0)$ and $B=\theta(-1,0,-2)$. Then
$$\mathrm{Mov}(S^{[3]})=\left\langle \tilde{H},\tilde{H}-\frac{1}{2}B\right\rangle.$$
\end{lem}
To understand the wall and chamber structure in $\mathrm{Mov}(S^{[3]})$, we apply \cite[Theorem 5.7]{BM14b}. The full list of walls is given in the table below.

\hspace*{3mm}
\begin{center}
    \begin{tabular}{|c|c|c|c|c|c|}
    \hline
    &&&&&\\[-6pt]
    $\Gamma$ & $\abf$ & $\abf^2$ & $(\vbf,\abf)$ & Wall& Type \\
    &&&&&\\[-6pt]
    \hline
    &&&&&\\[-6pt]
    $0$ & $(0,0,-1)$ & 0 & 1 &$x=0$ & divisorial  \\
    &&&&&\\[-6pt]
    \hline
    &&&&&\\[-6pt]
    $\frac12$ & $(1,-1,2)$ & $-2$ & 0 &$(x+2)^2+y^2=\left(\sqrt{2}\right)^2$& divisorial  \\[-6pt]
    &&&&&\\
    \hline
    \end{tabular}
    \vspace{0.2cm}
    \captionof{table}{Walls of $\mathrm{Mov}(S^{[3]})$} 
\end{center}
As a result, $S^{[3]}$ has no other nontrivial birational models. Moreover, there are no other walls with radii larger than $1$.

\subsection{Walls for $S^{[4]}$}
We have $\vbf=(1,0,-3)$.
\begin{lem}\cite[Proposition 13.1]{BM14b}
Let $\tilde{H}=\theta(0,-1,0)$ and $B=\theta(-1,0,-3)$. Then
$$\mathrm{Mov}(S^{[4]})=\left\langle \tilde{H},\tilde{H}-\frac{1}{2}B\right\rangle.$$
\end{lem}
The full list of walls is given in the table below.

\hspace*{3mm}
\begin{center}
    \begin{tabular}{|c|c|c|c|c|c|}
    \hline
    &&&&&\\[-6pt]
    $\Gamma$ & $\abf$ & $\abf^2$ & $(\vbf,\abf)$ & Wall& Type \\
    &&&&&\\[-6pt]
    \hline
    &&&&&\\[-6pt]
    $0$ & $(0,0,-1)$ & 0 & 1 &$x=0$ & divisorial  \\
    &&&&&\\[-6pt]
    \hline
    &&&&&\\[-6pt]
    $\frac {2}{5}$ & $(1,-1,2)$ & $-2$ & 1 & $\left(x+\frac{5}{2}\right)^2+y^2=\left(\frac{\sqrt{13}}{2}\right)^2$ & flop \\
    &&&&&\\[-6pt]
    \hline
    &&&&&\\[-6pt]
    $\frac 12$ & $(1,-1,1)$ & 0 & 2 & $(x+2)^2+y^2=1$ & divisorial  \\[-6pt]
    &&&&&\\
    \hline
    \end{tabular}
    \vspace{0.2cm}
    \captionof{table}{Walls of $\mathrm{Mov}(S^{[4]})$} 
\end{center}
As a result, there are two birational models of $S^{[4]}$. If we use ${}^{\sharp}S^{[4]}$ to denote the model not isomorphic to $S^{[4]}$, then ${}^{\sharp}S^{[4]}$ is obtained by performing a flop of $S^{[4]}$ along the locus $\{\xi_4\;|\;h^0(\Ical_{\xi_4}(1))\neq 0\}$.

\subsection{Walls for $S^{[8]}$} We let $\vbf=(1,0,-7)$.
\begin{lem}\cite[Proposition 13.1]{BM14b}
Let $\tilde{H}=\theta(0,-1,0)$ and $B=\theta(-1,0,-7)$. Then
$$\mathrm{Mov}(S^{[8]})=\left\langle \tilde{H},\tilde{H}-\frac{3}{8}B\right\rangle.$$
\end{lem}
The full list of walls is given in the table below.

\hspace*{3mm}
\begin{center}
    \begin{tabular}{|c|c|c|c|c|c|}
    \hline
    &&&&&\\[-6pt]
    $\Gamma$ & $\abf$ & $\abf^2$ & $(\vbf,\abf)$ & Wall & Type \\
    &&&&&\\[-6pt]
    \hline
    &&&&&\\[-6pt]
    $0$ & $(0,0,-1)$ & 0 & 1 & $x=0$ & divisorial \\
    &&&&&\\[-6pt]
    \hline
    &&&&&\\[-6pt]
    $\frac{2}{9}$ & $(1,-1,2)$ & $-2$ & 5 & $\left(x+\frac{9}{2}\right)^2+y^2=\left(\frac{\sqrt{53}}{2}\right)^2$ & flop \\
    &&&&&\\[-6pt]
    \hline
    &&&&&\\[-6pt]
    $\frac{1}{4}$ & $(1,-1,1)$ & 0 & 6 & $(x+4)^2+y^2=3^2$ & flop \\
    &&&&&\\[-6pt]
    \hline
    &&&&&\\[-6pt]
    $\frac{2}{7}$ & $(0,1,-7)$ & 2 & 7 & $\left(x+\frac{7}{2}\right)^2+y^2=\left(\frac{\sqrt{21}}{2}\right)^2$ & flop \\
    &&&&&\\[-6pt]
    \hline &&&&&\\[-6pt]
    $\frac{1}{3}$ & $(0,1,-6)$ & 2 & 6 & $(x+3)^2+y^2=\left(\sqrt{2}\right)^2$ & flop \\
    &&&&&\\[-6pt]
    \hline &&&&&\\[-6pt]
    $\frac{6}{17}$ & $(-1,3,-10)$ & $-2$ & 3 & $\left(x+\frac{17}{6}\right)^2+y^2=\left(\frac{\sqrt{37}}{6}\right)^2$ & flop \\
    &&&&&\\[-6pt]
    \hline &&&&&\\[-6pt]
    $\frac{4}{11}$ & $(1,-2,4)$ & 0 & 3 & $\left(x+\frac {11}{4}\right)^2+y^2=\left(\frac{3}{4}\right)^2$ & flop \\
    &&&&&\\[-6pt]
    \hline &&&&&\\[-6pt]
    $\frac{10}{27}$ & $(2,-5,13)$ & $-2$ & 1 & $\left(x+\frac {27}{10}\right)^2+y^2=\left(\frac{\sqrt{29}}{10}\right)^2$ & flop \\
    &&&&&\\[-6pt]
    \hline &&&&&\\[-6pt]
    $\frac{3}{8}$ & $(-1,3,-9)$ & 0 & 2 & $\left(x+\frac{8}{3}\right)^2+y^2=\left(\frac{1}{3}\right)^2$ & divisorial \\[-6pt]
    &&&&&\\
    \hline
    \end{tabular}
    \vspace{0.2cm}
    \captionof{table}{Walls of $\mathrm{Mov}(S^{[8]})$} 
\end{center}
As a result, there are eight birational models of $S^{[8]}$.

\subsection{Walls for $\vbf=(0,1,0)$} By \cite[Theorem 5.7]{BM14b}, there is no wall for $(0,1,0)$ whose radius is larger than $1$.

\subsection{Walls for $\vbf=(0,2,-2)$} By \cite[Theorem 5.3]{MZ16}, the only wall for $(0,2,-2)$ whose radius is larger than $1$ is a flopping wall given by
\begin{align*}
    \left(x+\frac{1}{2}\right)^2+y^2=\left(\frac{\sqrt{5}}{2}\right)^2.
\end{align*}

\subsection{Walls for $\vbf=(0,2,-1)$} These are computed in \cite[Section 5]{Hel21}. There are two walls with radii larger than $1$, given by
\begin{align*}
    \left(x+\frac{1}{4}\right)^2+y^2 & =\left(\frac{5}{4}\right)^2,\\
     \left(x+\frac{1}{4}\right)^2+y^2 & =\left(\frac{\sqrt{17}}{4}\right)^2.
\end{align*}

\end{document}